\documentclass[12pt]{amsart}
\usepackage{amsmath}
\usepackage{amsthm}
\usepackage{amssymb}
\usepackage{amscd}
\usepackage{amsfonts}
\usepackage{amsbsy}
\usepackage{epsfig,afterpage}
\usepackage[dvips]{psfrag}
\usepackage{subfigure}
\usepackage{srcltx}

\usepackage{verbatim}
\usepackage{graphicx}

\usepackage{multirow}
\usepackage{graphics}
\usepackage{epstopdf}
\usepackage{overpic}
\usepackage{multicol}

\usepackage[colorlinks=true, linkcolor=blue, citecolor=red, urlcolor=blue]{hyperref}

\newtheorem {theorem} {Theorem}
\newtheorem {proposition} [theorem]{Proposition}

\newtheorem {remark} [theorem]{Remark}

\newtheorem {definition} [theorem]{Definition}

\newtheorem{mtheorem}{Theorem}

\usepackage{float}
\usepackage{tikz, wrapfig}
\tikzset{node distance=3cm, auto}
\allowdisplaybreaks

\newtheorem{assumption}{Assumption}

\begin{document}

\title[Cyclicity of sliding cycles in regularized PWL systems]
{Cyclicity of sliding cycles in regularizations of piecewise linear two-folds}

\author[R. Huzak, K.U. Kristiansen, O.H. Perez and A.L.M.T da Silva]
{Renato Huzak$^{1}$, Kristian Uldall Kristiansen$^{2}$, \\
Otavio Henrique Perez$^{3}$, Andr\'e Luis Martins Tomaz da Silva$^{1,4}$}

\address{$^{1}$Hasselt University, Campus Diepenbeek, Agoralaan Gebouw D, 3590 Diepenbeek, Belgium}

\address{$^{2}$Department of Applied Mathematics and Computer Science, Technical University of Denmark, 2800
Kgs. Lyngby, Denmark}

\address{$^{3}$University of S\~{a}o Paulo (USP), Institute of Mathematics and Computer Science. Avenida Trabalhador S\~{a}o Carlense, 400, Zip Code 13566-590, S\~{a}o Carlos, S\~{a}o Paulo, Brazil}

\address{$^{4}$Universidade de Brasília, Departamento de Matemática, Campus Universitário Darcy Ribeiro, Asa Norte 70910-900, Brasília-DF, Brazil}

\email{renato.huzak@uhasselt.be}
\email{krkri@dtu.dk}
\email{otavio.perez@icmc.usp.br}
\email{andreluis.martinstomazdasilva@uhasselt.be, andreluis@mat.unb.br}

\thanks{ .}


\date{}
\dedicatory{}


\begin{abstract}
We study limit cycles produced by Sotomayor–Teixeira regularizations of piecewise linear vector fields with generic two-fold singularities. We focus on the visible-visible and visible-invisible cases where sliding cycles occur, treating cycles from both sides of the switching manifold in a unified way. In further details, by relating the cyclicity of sliding cycles and zeros of the slow divergence integral, we prove that the cyclicity of compact families of sliding cycles is bounded by two when such integral does not vanish identically. In contrast to previous works, we do not perform a case-by-base study, but instead relate zeros of slow divergence integrals to crossing limit cycles of a suitably defined auxiliary piecewise linear system. For visible folds, we provide necessary and sufficient conditions that assure the existence of a unique simple zero of the slow divergence integral, which implies the existence of two limit cycles. We also show that, when a hyperbolic singularity of the sliding vector field lies at the boundary of the sliding segment, then the cyclicity is bounded by one.
\end{abstract}

\maketitle

\section{Introduction and statement of the problem}







The second part of Hilbert’s 16th problem \cite{Hil1900} has been a central topic in the qualitative theory of differential equations for more than a century. It asks whether there exists a finite uniform upper bound on the number of limit cycles of planar polynomial vector fields of a given degree $n$. Despite considerable effort and significant progress, this problem remains open even for $n=2$ (see, for instance, \cite{DRR94}). 
For results related to Smale’s version of the problem \cite{Smale2000}, where one restricts attention to polynomial Li\'enard equations, we refer to \cite{Cau2012,DMH2015,DPR2007,HDM-generalized}.

Recently, considerable attention has been devoted to a related problem inspired by the second part of Hilbert’s 16th problem, namely, whether there exists a uniform upper bound on the number of crossing limit cycles of piecewise polynomial vector fields. When considering piecewise \textit{linear} (PWL) vector fields, using the integral characterization of Poincar\'e half-maps \cite{CFS2021}, it was proved in \cite{CFSN2023iii} that the maximum number of crossing limit cycles is uniformly bounded by $8$. To the best of our knowledge, $3$ crossing limit cycles have been found (see, for instance, \cite{BPT2013, GRS2024,HY2012,LliPon2013}). It is also worth mentioning the continuous case, whose uniform upper bound for the number of crossing limit cycles is $1$, and the bound is reached. This was conjectured by Lum and Chua in \cite{LumChu1991} and proved for the first time in \cite{FPRT1998}. The first case-independent proof of Lum--Chua's conjecture was given in \cite{CFSN2021} and, in this last reference, the approach relied on the integral characterization of Poincar\'e half-maps \cite{CFS2021}.

This paper addresses a related problem inspired by the second part of Hilbert’s 16th problem, which is the search for a uniform upper bound on the number of limit cycles of regularized piecewise polynomial systems. In \cite{HK23} the authors considered the \textit{Sotomayor--Teixeira} regularization method \cite{ST96} and they proved that there exists a piecewise polynomial vector field $Z$, formed by one quadratic and one linear vector field, such that, for a given integer $k \geq 1$, there exists a regularization function $\phi_{k}$ such that the regularization of $Z$ has $k+1$ hyperbolic limit cycles. Thus, for piecewise quadratic systems, there is no uniform upper bound on the number of limit cycles of their regularizations. It remains an open problem whether there exists a uniform upper bound on the number of limit cycles in Sotomayor--Teixeira regularizations of PWL systems, see \cite{HK24}. The present paper addresses this problem for sliding cycles associated with the generic $VV_1$ and $VI_3$ configurations, in the case where the slow divergence integral does not vanish identically. For other types of regularizations and upper bounds on the number of limit cycles, we refer to \cite{DMHP,HP}.

In \cite{HK24} the authors considered regularizations of PWL vector fields having generic visible-invisible \textit{two-folds}. Two-folds can be classified into three types: visible-visible, visible-invisible, and invisible-invisible (which we  denote by $VV$, $VI$ and $II$, respectively). Following the notation of \cite{KRG03}, each configuration is divided into several subcases, denoted by $VV_{1,2}$, $VI_{1,2,3}$, and $II_{1,2}$. This classification depends first on whether \textit{sliding} regions are present, as in the cases $VV_{1}$, $VI_{2}$, $VI_{3}$, and $II_{1}$, or absent, as in $VV_{2}$, $VI_{1}$, and $II_{2}$. When sliding occurs, one must also take into account the dynamics of the \textit{Filippov sliding vector field} and whether the unfolding gives rise to singularities of this vector field. We refer to Fig. \ref{fig-7-genric-cases} in Section \ref{sec-psvf} and \cite{GST2011,KRG03} for further details on bifurcations of generic two-folds.

The aim of \cite{HK24} was to initiate a systematic study of the cyclicity of sliding cycles (that is, the maximal number of limit cycles that can bifurcate from them) in Sotomayor--Teixeira regularizations of PWL systems exhibiting a two-fold singularity of type $VI_{3}$. The sliding cycles considered there lie in the half-plane containing the invisible fold, and it was proved that at most two limit cycles can bifurcate from such cycles. The treatement of the $VI_{3}$ case in \cite{HK24} was based on a case-by-case analysis. The main contribution of the present paper is to provide a unified, case-independent treatment of both the $VV_{1}$ and $VI_{3}$ configurations; see Fig. \ref{fig-sliding-cycles} and Section \ref{sec-main-cyclicity} for the precise definitions of $VV_{1}$ and $VI_{3}$. The sliding cycles shown in Fig. \ref{fig-sliding-cycles} contain both stable and unstable sliding segments. They are sometimes referred to as canard cycles, see \cite{HK23}, since their geometry resembles that of classical canard cycles.

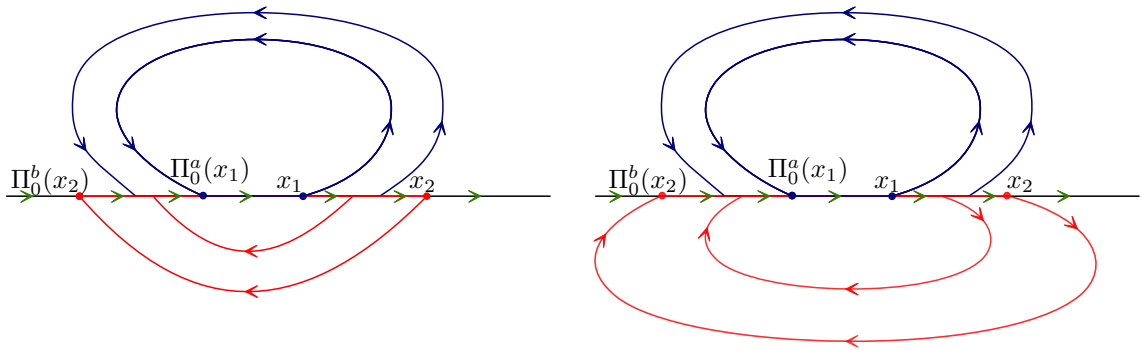
\begin{figure}[htb]
	\begin{center}
		\includegraphics[width=15cm]{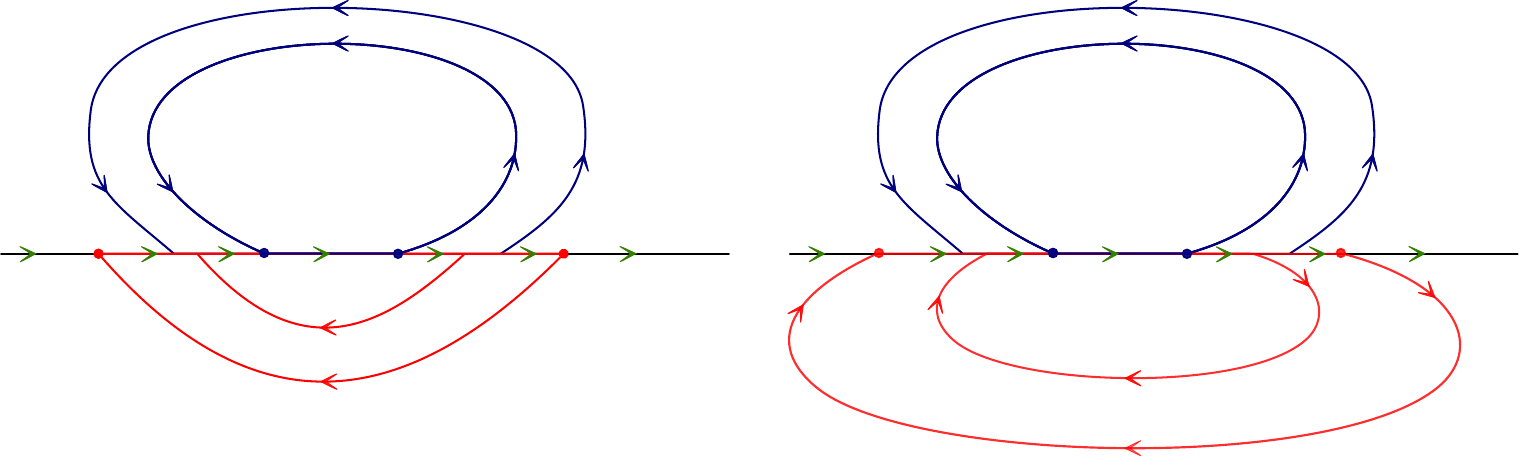}
       {\footnotesize \put(-100,60){$x_{1}$}
        \put(-140,64){$\Pi^{a}_{0}(x_{1})$}
\put(-325,60){$x_{1}$}
        \put(-365,64){$\Pi^{a}_{0}(x_{1})$}
     \put(-50,60){$x_{2}$}
        \put(-200,60){$\Pi^{b}_{0}(x_{2})$}
        \put(-275,60){$x_{2}$}
        \put(-425,60){$\Pi^{b}_{0}(x_{2})$}}
         \end{center}
	\caption{Sliding cycles addressed in this paper. Two-folds $VI_{3}$ (left) and $VV_{1}$ (right). The functions $\Pi^{a,b}_{0}$ stand for Poincar\'e half-maps from above and below, respectively. Sliding cycles from above and below are highlighted in blue and red, respectively.}
	\label{fig-sliding-cycles}
\end{figure}

The cases $VV_{1}$ and $VI_{3}$ are the only generic two-fold configurations in which sliding cycles here can occur, see again Fig. \ref{fig-7-genric-cases}. Thus, the present paper covers all such geometric sliding-cycle configurations, while the degenerate case in which the slow divergence integral vanishes identically is left for future work. 
Our approach combines techniques from geometric singular perturbation theory \cite{DMDR,HK23}, such as family blow-ups and slow divergence integrals, and integral characterization of Poincar\'e half-maps \cite{CFS2021}. The key insight is to relate the existence of zeros of the slow divergence integral to the existence of crossing limit cycles of a suitably defined auxiliary piecewise linear vector field. Such relationship is based upon the integral characterization of Poincar\'e half-maps developed in \cite{CFS2021}. This approach is based upon an idea suggested by an anonymous referee of \cite{HK24}.

As illustrated in Fig.~\ref{fig-sliding-cycles}, one may study the cyclicity of sliding cycles from both above and below. We denote these cycles by $\Gamma^{a}_{x}$ and $\Gamma^{b}_{x}$, respectively, where $x>0$. Our main results in this direction are Theorems~\ref{theo-cycicity}--\ref{mthm-singularity}, stated precisely in Sections~\ref{sec-main-cyclicity} and~\ref{sec-cycl-sing}. We postpone the technical details to later sections and first briefly outline the main contributions of the paper.

Theorem \ref{theo-cycicity} provides criteria, formulated in terms of the zeros of the slow divergence integral, for obtaining upper bounds on the number of limit cycles bifurcating from sliding cycles and for establishing their existence, both from above and from below. Similar criteria were used in the earlier works \cite{HK23,HK24}, see also \cite{DMH2015}. Roughly speaking, if the slow divergence integral has $k$ simple zeros, then the regularized system associated with either the $VV_{1}$ or the $VI_{3}$ case has $k+1$ hyperbolic limit cycles. We emphasize that Theorem \ref{theo-cycicity} is not restricted to regularized PWL systems, but applies more generally to regularized piecewise smooth systems with a two-fold singularity of type $VV_{1}$ or $VI_{3}$.
A proof of Theorem \ref{theo-cycicity} is given in Appendix \ref{sec-blow-ups}.

Theorems \ref{thm-1}--\ref{mthm-singularity} concern regularized PWL systems with a two-fold singularity of type $VV_{1}$ or $VI_{3}$. In Theorem~\ref{thm-1}, we show that each slow divergence integral, from above or below, associated with the sliding cycles shown in Fig.~\ref{fig-sliding-cycles} either vanishes identically or has at most one zero and, if such zero exists, it is simple. The proof of Theorem \ref{thm-1} is given in Section \ref{sec-zeros-sdi}. 

Theorem \ref{mhtm-c}, which follows directly from Theorems \ref{theo-cycicity} and \ref{thm-1}, states that if the slow divergence integral, from above or below, has at most one zero counted with multiplicity, then the regularized PWL system has at most two hyperbolic limit cycles bifurcating from the corresponding sliding cycles. Theorem \ref{mthm-existence-region} concerns sliding cycles associated with visible folds and  provides necessary and sufficient conditions for the corresponding slow divergence integral to have a unique simple zero. Combined with  Theorem \ref{theo-cycicity}, this yields the existence of two hyperbolic limit cycles.

Finally, for completeness, in Section~\ref{sec-cycl-sing} we state Theorem~\ref{mthm-singularity}, which concerns the cyclicity of sliding cycles when the sliding vector field has a singularity at one of the corner points; see Fig. \ref{fig:centerandfocus}. In this case, the cyclicity is at most one. The slow divergence integral associated with such sliding cycles is not well defined, and hence the case-independent approach based on \cite{CFS2021} cannot be applied. Instead, one must study the full divergence integral of the regularized PWL system. We refer to \cite{HHPY26} for the treatment of sliding cycles from below in the $VI_{3}$ case. The same approach applies to sliding cycles from above in the $VI_{3}$ case and to sliding cycles from both above and below in the $VV_{1}$ case. A unified proof of these remaining cases is given in Appendix \ref{sec-blow-ups}.

This paper is organized as follows. In Section~\ref{sec-psvf}, we recall the definitions and tools needed to state and prove our main results, including piecewise smooth vector fields, regularizations, sliding regions and Poincar\'e half-maps. The main results of the paper, Theorems \ref{theo-cycicity}--\ref{mthm-existence-region}, are stated in Section~\ref{sec-main-cyclicity}. Section~\ref{sec-pwl} is devoted to the study of PWL systems, while the proofs of Theorems \ref{thm-1} and \ref{mthm-existence-region} are given in Section \ref{sec-zeros-sdi}. In Section \ref{sec-cycl-sing}, we state Theorem~\ref{mthm-singularity}. Finally, Theorems \ref{theo-cycicity} and \ref{mthm-singularity} are proved in Appendix \ref{sec-blow-ups}.

\section{Background on piecewise smooth vector fields and regularizations}\label{sec-psvf}

\subsection{Piecewise smooth vector fields and Filippov's convention}

We define a \textit{piecewise smooth vector field} as 
\begin{equation}\label{eq-pwl-smooth}
Z_{\lambda}(x,y) = \left\{
  \begin{array}{rlc}
   Z^{+}_{\lambda}(x,y), & \text{if} & h(x,y) > 0,  \\
   Z^{-}_{\lambda}(x,y), & \text{if} & h(x,y) < 0,
  \end{array}
\right.
\qquad \lambda \sim 0,
\end{equation}
where $Z^{\pm}_{\lambda}(x,y) := Z^{\pm}(x,y,\lambda)$ and $Z_{\lambda}(x,y) := Z(x,y,\lambda)$. In a neighborhood $U\subset\mathbb{R}^{2}$ of $p\in\mathbb{R}^{2}$, the vector fields $Z^{\pm}_{\lambda}$ are $C^{\infty}$-smooth with respect to $(x,y)$ and the parameter $\lambda$, and $h:U\rightarrow \mathbb{R}$ is a $C^{\infty}$-smooth function. We further denote $Z^{\pm}_{\lambda} = (X^{\pm}_{\lambda},Y^{\pm}_{\lambda})$, and the set $\Sigma = \{(x,y)\in U; h(x,y) = 0\}$ is called \textit{switching manifold} or \textit{switching set}. In Section \ref{sec-pwl} we consider the class of \textit{piecewise linear vector fields} (PWL vector fields for short), where we assume that $Z^{\pm}_{\lambda}$ and $h$ are each affine with respect to $(x,y)$.

The \textit{Lie derivative} of $h$ with respect to the vector field $Z^{\pm}_{\lambda}$ is given by $\mathcal{L}_{Z^{\pm}_{\lambda}}(h) = \nabla h \cdot Z^{\pm}_{\lambda}$, where $\nabla h$ is the gradient vector of $h$ and the dot $\cdot$ denotes the usual scalar product. The second order Lie derivative is given by $(\mathcal{L}_{Z^{\pm}_{\lambda}})^{2}(h) = \mathcal{L}_{Z^{\pm}_{\lambda}}(\mathcal{L}_{Z^{\pm}_{\lambda}}(h)) = \nabla(\mathcal{L}_{Z^{\pm}_{\lambda}}(h)) \cdot Z^{\pm}_{\lambda}$. Lie derivatives allow us to define the following subsets of $\Sigma$.

\begin{itemize}
    \item[(1)] The \textit{sliding subset} (or \textit{sliding region}) is the set
    $$\Sigma^{sl}_{\lambda} := \{q\in\Sigma\,\vert \, \mathcal{L}_{Z^{+}_{\lambda}}(h)(q)\mathcal{L}_{Z^{-}_{\lambda}}(h)(q) < 0\}.$$
This set can be further decomposed into the \textit{attracting} (or \textit{stable}) \textit{sliding subset}, consisting of all points $q\in\Sigma$ such that $\mathcal{L}_{Z^{+}_{\lambda}}(h)(q) < 0$ and $\mathcal{L}_{Z^{-}_{\lambda}}(h)(q) > 0$; and the \textit{repelling} (or \textit{unstable}) \textit{sliding subset}, consisting of all points $q\in\Sigma$ such that $\mathcal{L}_{Z^{+}_{\lambda}}(h)(q) > 0$ and $\mathcal{L}_{Z^{-}_{\lambda}}(h)(q) < 0$.  
    \item[(2)] The \textit{crossing subset} (or \textit{crossing region}) is the set
    $$\Sigma^{cr}_{\lambda} := \{q\in\Sigma\,\vert \, \mathcal{L}_{Z^{+}_{\lambda}}(h)(q)\mathcal{L}_{Z^{-}_{\lambda}}(h)(q) > 0\}.$$
    \item[(3)] The set of \textit{tangency points} is the set
    $$\Sigma^{T}_{\lambda} := \{q\in\Sigma\,\vert \, \mathcal{L}_{Z^{+}_{\lambda}}(h)(q) = 0 \ \text{or} \ \mathcal{L}_{Z^{-}_{\lambda}}(h)(q) = 0\}.$$
    Observe that if $\mathcal{L}_{Z^{+}_{\lambda}}(h)(q) = 0$, the trajectory of $Z^{+}_{\lambda}$ through $q$ is tangent to $\Sigma$. We say that $q$ is a \textit{tangency point from above}. Analogously, if $\mathcal{L}_{Z^{-}_{\lambda}}(h)(q) = 0$, then the trajectory of $Z^{-}_{\lambda}$ through $q$ is tangent to $\Sigma$ and we call it \textit{tangency point from below}.
\end{itemize}

Along $\Sigma^{cr}_{\lambda}$, trajectories can be extended from $Z^{+}_{\lambda}$ to $Z^{-}_{\lambda}$ or vice versa by concatenating orbits of $Z^{+}_{\lambda}$ and $Z^{-}_{\lambda}$. On the other hand, trajectories of $Z^{\pm}_{\lambda}$ reach $\Sigma^{sl}_{\lambda}$ in finite time and, to be able to continue trajectories, a vector field has to be defined along $\Sigma^{sl}_{\lambda}$. For this purpose, we adopt the so called \textit{Filippov's convention} \cite{Filippov}, that is, we consider the \textit{sliding vector field} on $\Sigma^{sl}_{\lambda}$. More precisely, the sliding vector field $Z^{sl}_{\lambda}:\Sigma^{sl}_{\lambda}\rightarrow T\Sigma$ is given by

$$Z^{sl}_{\lambda}(q):= \displaystyle\frac{1}{\mathcal{L}_{Z^{+}_{\lambda}}(h)(q) - \mathcal{L}_{Z^{-}_{\lambda}}(h)(q)}\Big{(}\mathcal{L}_{Z^{+}_{\lambda}}(h)(q)Z^{-}_{\lambda}(q) - \mathcal{L}_{Z^{-}_{\lambda}}(h)(q)Z^{+}_{\lambda}(q)\Big{)}.$$

\subsection{Folds and two-folds}\label{sec-psvf-twofolds}

In this paper we pay attention to tangency points of quadratic order, and they will be called \textit{fold points}. A point $q\in\Sigma$ is a \textit{fold point of $Z^{+}_{\lambda}$} (or \textit{fold point from above}) if it satisfies
\begin{equation*}
\begin{array}{lcl}
Z^{+}_{\lambda}(q) & \neq & 0, \\
\mathcal{L}_{Z^{+}_{\lambda}}(h)(q) & = & 0, \\
(\mathcal{L}_{Z^{+}_{\lambda}})^2(h)(q) & \neq & 0.
\end{array}
\end{equation*}

The fold point $q\in\Sigma^{T}_{\lambda}$ is \textit{visible} if $(\mathcal{L}_{Z^{+}_{\lambda}})^2(h)(q) > 0$ and \textit{invisible} if $(\mathcal{L}_{Z^{+}_{\lambda}})^2(h)(q) < 0$. A \textit{fold point of $Z^{-}_{\lambda}$} (or \textit{fold point from below}) is defined in a similar fashion, however, in this case, the fold point $q\in\Sigma^{T}_{\lambda}$ is visible if $(\mathcal{L}_{Z^{-}_{\lambda}})^2(h)(q) < 0$ and invisible if $(\mathcal{L}_{Z^{-}_{\lambda}})^2(h)(q) > 0$.

When $q\in\Sigma^{T}_{\lambda}$ is a fold point of both $Z^{\pm}_{\lambda}$, it will be called \textit{two-fold}. Generic two-folds are known to present codimension $1$ bifurcations, and their bifurcation diagrams can be found in  \cite{GST2011,KRG03}. Two-folds can be classified into  three types: visible-visible, visible-invisible, and invisible-invisible, which are simply denoted by $VV$, $VI$ and $II$, respectively. Following the notation adopted in \cite{KRG03}, for each case, there are subcases (denoted by $VV_{1,2}$, $VI_{1,2,3}$ and $II_{1,2}$) that must be considered depending on (a) whether there is sliding ($VV_{1}$, $VI_{2}$, $VI_{3}$ and $II_{1}$) or not ($VV_{2}$, $VI_{1}$ and $II_{2}$) and, if there is sliding, one should take into account (b) the direction of $Z^{sl}_{\lambda}$ and finally (c) whether the unfolding leads to singularities of the sliding vector field $Z^{sl}_{\lambda}$ on $\Sigma^{sl}_{\lambda}$. We refer to Fig. \ref{fig-7-genric-cases}.

\begin{figure}[htb]
	\begin{center}
	\begin{subfigure}
    \centering
    \includegraphics[width=5cm]{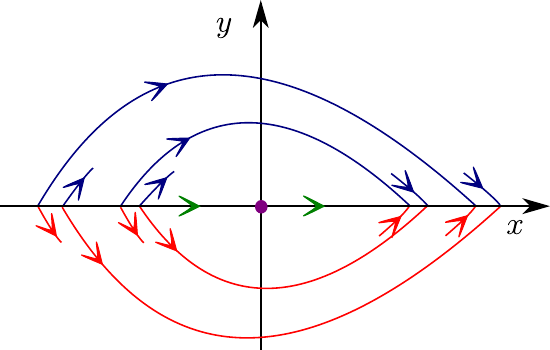}
    {\footnotesize \put(-2.99cm,-0.5cm){$II_1$}}
    \end{subfigure}    
  \hspace{1cm} 
  \begin{subfigure}
    \centering    \includegraphics[width=5cm]{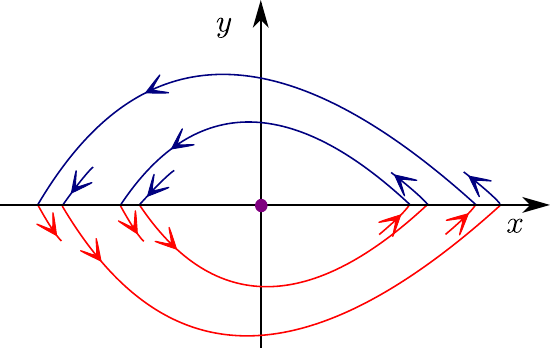}
    {\footnotesize \put(-2.99cm,-0.5cm){$II_2$}}
  \end{subfigure}
  \vspace{0,5cm}
  \vfill
  \begin{subfigure}
    \centering    \includegraphics[width=5cm]{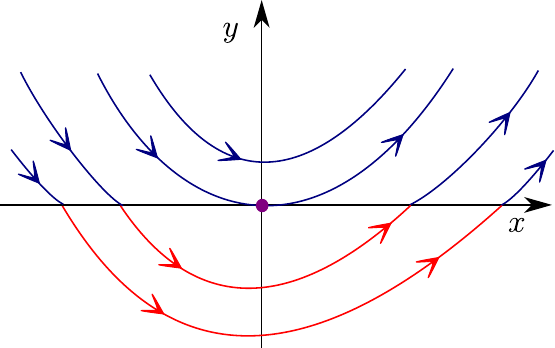}
    {\footnotesize \put(-2.99cm,-0.5cm){$VI_1$}}   
  \end{subfigure}  
  \hspace{1cm}
  \begin{subfigure}
    \centering    \includegraphics[width=5cm]{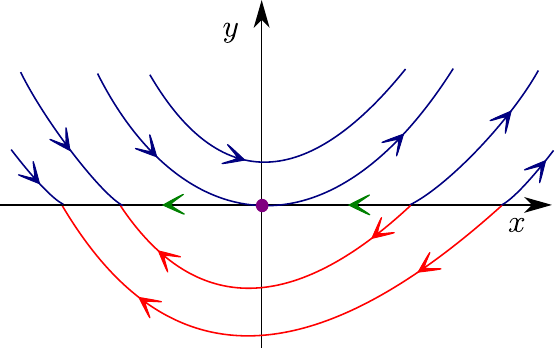}
    {\footnotesize \put(-2.99cm,-0.5cm){$VI_2$}}
  \end{subfigure}
  \vspace{0,5cm}
  \vfill
  \begin{subfigure}
    \centering    \includegraphics[width=5cm]{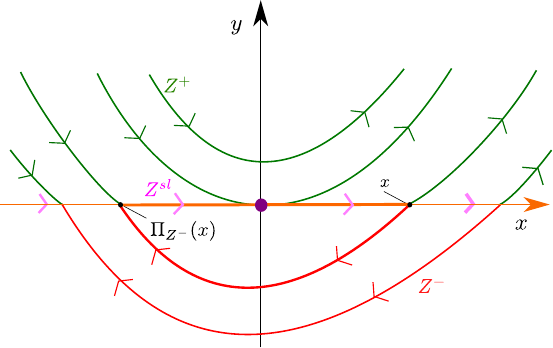}
    {\footnotesize \put(-2.99cm,-0.5cm){$VI_3$}}
  \end{subfigure}
  \hspace{1cm}
  \begin{subfigure}
    \centering    \includegraphics[width=5cm]{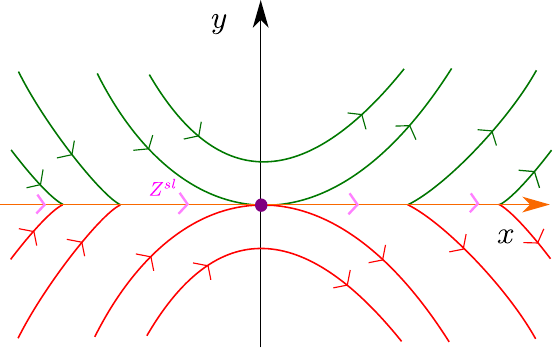}
    {\footnotesize \put(-2.99cm,-0.5cm){$VV_1$}}
  \end{subfigure}
  \vspace{0,5cm}
  \vfill
  \begin{subfigure}
    \centering    \includegraphics[width=5cm]{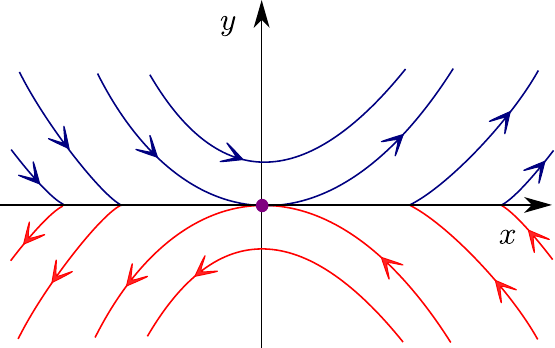}
    {\footnotesize \put(-2.99cm,-0.5cm){$VV_2$}}
  \end{subfigure}
         \end{center}
	\caption{Different types of two-folds, following the notation of \cite{KRG03}.}
	\label{fig-7-genric-cases}
\end{figure}

In this paper, we consider the cases $VV_{1}$ and $VI_{3}$. Geometrically speaking, both cases are characterized by the fact that, in a neighborhood $U$ of the two-fold $q$, the sliding vector field $Z^{sl}_{\lambda}$ points from the attracting to the repelling branch of $\Sigma^{sl}_{\lambda}$. We refer the reader to Fig. \ref{fig-sliding-cycles} and Section \ref{sec-main-cyclicity} for the precise definitions of $VV_{1}$ and $VI_{3}$. Cases $VV_{1}$ and $VI_{3}$ are the only cases whose regularization can produce canard limit cycles near generic two-folds. Observe that, in cases $VI_{2}$ and $II_{1}$, near the two-fold $q$ the sliding vector field $Z^{sl}_{\lambda}$ points from the repelling to the attracting branch of $\Sigma^{sl}_{\lambda}$. 


\subsection{Poincar\'e half-maps}

This section is devoted to define Poincar\'e half-maps of the vector fields $Z^{\pm}_{\lambda}$. Firstly, consider the vector field $Z^{+}_{\lambda}$. Given a point $p_{0}\in\Sigma$, denote the solution of the system of ODEs related to $Z^{+}_{\lambda}$ by $p_{0}$ as $\Psi^{+}_{\lambda}(t;p_{0})$, and observe that $\Psi^{+}_{\lambda}(0;p_{0}) = p_{0}$. Now, suppose that there exists a real positive number $\tau = \tau(p_{0}) < +\infty$ such that $\Psi^{+}_{\lambda}(\tau; p_{0}) = p_{1}\in\Sigma$ and, for all time $t\in(0,\tau)$, it holds the inclusion $\{\Psi^{+}_{\lambda}(t; p_{0}); 0 < t < \tau\}\subset \{(x,y)\in U; h(x,y) > 0\}$. Under these assumptions, the function that maps $p_{0}$ to $p_{1}$ is called \textit{Poincar\'e half-map} of $Z^{+}_{\lambda}$ and will be denoted by $\Pi^{a}_{\lambda}$. The positive and finite value $\tau$ is called \textit{flight time}. The Poincar\'e half-map of $Z^{-}_{\lambda}$ is defined in the same fashion and will be denoted by $\Pi^{b}_{\lambda}$. We choose the notation $\Pi^{a,b}_{\lambda}$ for the Poincar\'e half-maps of $Z^{\pm}_{\lambda}$, where $a$ stands for ``above'' and $b$ stands for ``below''.

\subsection{Sliding cycles}

A \textit{sliding cycle} of the piecewise smooth vector field $Z_{\lambda}$ is a closed curve given by the concatenation of compact pieces of the sliding subset $\Sigma^{sl}_{\lambda}$ and regular orbits of $Z^{+}_{\lambda}$ or $Z^{-}_{\lambda}$, in such a way that the orientation of $Z^{sl}_{\lambda}$ and $Z^{\pm}_{\lambda}$ agree. In some references, sliding cycles are called \textit{poly-trajectories} (see for instance \cite{SM2002}). The sliding cycles that will be addressed in this paper will be defined  precisely in Section \ref{sec-main-cyclicity}. 

\subsection{Regularizations of piecewise smooth vector fields}\label{sec-psvf-reg}

Let $\phi:\mathbb{R}\rightarrow (0,1)$ be a $C^{\infty}$-smooth function satisfying the following assumptions.
\begin{assumption}\label{assump-reg1}
For $s\to\pm\infty$, the function $\phi$ satisfies
$$
\phi(s)\to 1 \text{ if }  s\to\infty, \quad \text{ and } \quad \phi(s)\to 0 \text{ if }  s\to-\infty.
$$    
\end{assumption}
\begin{assumption}\label{assump-reg2}
The function $\phi$ is strictly monotone, that is,
$\phi'(s) > 0$ for all $s\in\mathbb{R}$.    
\end{assumption}
\begin{assumption}\label{assump-reg3}
The function $\phi$ is \textit{smooth at $\pm\infty$} in the following sense: both functions
$$
\phi_{+}(s) :=
\begin{cases}
1, & \text{for } s=0,\\
\phi(s^{-1}), & \text{for } s>0,
\end{cases}
\qquad 
\phi_{-}(s):=
\begin{cases}
\phi(-s^{-1}), & \text{for } s>0,\\
0, & \text{for } s=0,
\end{cases}
$$
are $C^\infty$-smooth at $s=0$.    
\end{assumption}

If $\phi$ satisfies assumptions \ref{assump-reg1}, \ref{assump-reg2} and \ref{assump-reg3}, then it will be called a \textit{regularization function}. Given a regularization function $\phi$, a \textit{regularization} of the piecewise smooth vector field $Z_{\lambda}$ is the $C^{\infty}$-smooth vector field
\begin{equation}\label{eq-psvf-reg}
Z_{\varepsilon,\lambda}^{\phi}(x,y) = Z^{+}_{\lambda}(x,y)\phi\left( \frac{h(x,y)}{\varepsilon^{2}}\right)
+ Z^{-}_{\lambda}(x,y)\left(1-\phi\left(\frac{h(x,y)}{\varepsilon^{2}}\right)\right), 
\end{equation}
where $0 < \varepsilon \ll 1$ and $\lambda \sim 0$. We also say that $Z_{\varepsilon, \lambda}^{\phi}$ is the \textit{regularized vector field}. Observe that $Z_{\varepsilon,\lambda}^{\phi} \rightarrow Z_{\lambda}$ uniformly on compact subsets of $U\backslash \Sigma$ (recall that the open set $U\subset\mathbb{R}^{2}$ is the domain of $Z_{\lambda}$) when $\varepsilon\rightarrow 0$ (see \cite{PanSil2017}). Moreover, we keep the superscript $\phi$ in the notation $Z_{\varepsilon,\lambda}^{\phi}$ in order to stress that the regularization strongly depends on the function $\phi$. Indeed, different $\phi$ may lead to different dynamics of the regularized vector field $Z_{\varepsilon,\lambda}^{\phi}$ (see for instance \cite{HK23}). The regularization process adopted in this paper is due to Sotomayor and Teixeira \cite{ST96} and we refer to \cite{DMHP, HP, PanSil2017, PRS2023} and the references therein for other regularization processes.

It was observed in \cite{BST06} that the system of ODE's associated to the regularized vector field \eqref{eq-psvf-reg} is a \textit{slow-fast system}. In addition, the authors proved that the existence of sliding subsets associated to the piecewise smooth vector field \eqref{eq-pwl-smooth} implies the existence of local invariant manifolds of \eqref{eq-psvf-reg}, whose reduced flow is a regular perturbation of the sliding vector field $Z^{sl}_{\lambda}$.

\section{Statement of the Main Results}\label{sec-main-cyclicity}

In this section and in the rest of the paper, we consider a coordinate system such that the smooth function $h$ is given by $h(x,y) = y$ (that is, the switching manifold is given by $\Sigma = \{y = 0\}$) and the two-fold of $Z_{0}$ for $\lambda = 0$ (which is either $VV_{1}$ or $VI_{3}$) is positioned at the origin.

Following \cite{KRG03}, we say that $Z_{0}$ has a two-fold of type $VV_{1}$ at the origin $(x,y)=(0,0)$ if 
\begin{equation}\label{eq-prop-vv1}
\left\{
\begin{array}{rcl}
X^{+}_{0}(0,0) & > &  0, \\
Y^{+}_{0}(0,0) & = & 0, \\
\displaystyle\frac{\partial Y^{+}_{0}}{\partial x}(0,0) & > & 0,
\end{array}
\right.
\qquad
\left\{
\begin{array}{rcl}
X^{-}_{0}(0,0) & > & 0, \\
Y^{-}_{0}(0,0) & = & 0, \\
\displaystyle\frac{\partial Y^{-}_{0}}{\partial x}(0,0) & < & 0.
\end{array}
\right.
\end{equation}

Similarly, we say that $Z_{0}$ has a two-fold of type $VI_{3}$ at the origin $(x,y)=(0,0)$ if 
\begin{equation}\label{eq-prop-vi3}
\left\{
\begin{array}{rcl}
X^{+}_{0}(0,0) & > &  0, \\
Y^{+}_{0}(0,0) & = & 0, \\
\displaystyle\frac{\partial Y^{+}_{0}}{\partial x}(0,0) & > & 0,
\end{array}
\right.
\qquad
\left\{
\begin{array}{rcl}
X^{-}_{0}(0,0) & < & 0, \\
Y^{-}_{0}(0,0) & = & 0, \\
\displaystyle\frac{\partial Y^{-}_{0}}{\partial x}(0,0) & < & 0.
\end{array}
\right.
\end{equation} and
\begin{equation}\label{VI3-extra-condition}
    \left(\frac{\partial Y^{+}_{0}}{\partial x}X^{-}_{0} - \frac{\partial Y^{-}_{0}}{\partial x}X^{+}_{0}\right)(0,0)>0.
\end{equation}

The goal of this paper is to provide cyclicity results of regularized PWL vector fields near two-folds, and our Theorems and proofs concern a unified approach for both cases $VV_{1}$ and $VI_{3}$. Observe that, when $\Sigma = \{y = 0\}$, the sliding vector field is given by
\begin{equation}\label{eq-sliding-x}
Z^{sl}_{\lambda}(x,0) =: \left(X^{sl}_{\lambda}(x),0\right) = \left(\displaystyle\frac{\operatorname{det}Z_{\lambda}}{Y^{+}_{\lambda} - Y^{-}_{\lambda}}(x,0),0\right), \ \operatorname{det}Z_{\lambda}(x) = \left(Y^{+}_{\lambda}X^{-}_{\lambda} - Y^{-}_{\lambda}X^{+}_{\lambda}\right)(x,0),
\end{equation}
where the first equality defines $X^{sl}_{\lambda}$. Due to Equation \eqref{eq-sliding-x}, in this section and in the rest of this paper we refer to $X^{sl}_{\lambda}$ as the sliding vector field. Straightforward computations using the definitions of $VV_{1}$ and $VI_{3}$ lead us to the following proposition.

\begin{proposition}\label{prop-vv1}
Assume that the piecewise smooth vector field $Z_{0}$ has a two-fold of type $VV_{1}$ or $VI_{3}$ at the origin $(x,y)=(0,0)$. 
Then, there exist constants $\delta_{1} < 0$ and $\delta_{2}> 0$ such that the following statements hold for $\lambda = 0$.
\begin{itemize}
    \item[(1)] The intervals $[\delta_{1},0)\subset\Sigma^{sl}_{0}$ and $(0,\delta_{2}]\subset\Sigma^{sl}_{0}$ correspond to the stable and unstable sliding regions, respectively.
    \item[(2)] The sliding vector field $X^{sl}_{0}$ is well defined in $[\delta_{1},\delta_{2}]$ and $X^{sl}_{0}(x) > 0$, for all $x\in[\delta_{1},\delta_{2}]$.
    \item[(3)] It holds $\frac{\partial Y^{+}_{0}}{\partial x}(0,0) - \frac{\partial Y^{-}_{0}}{\partial x}(0,0) > 0$.
\end{itemize}
\end{proposition}
\begin{proof}
The proposition follows from \cite{KRG03}. We nevertheless include the proof for completeness. Assume that either \eqref{eq-prop-vv1} holds, corresponding to the case $VV_{1}$, or that \eqref{eq-prop-vi3} together with \eqref{VI3-extra-condition} holds, corresponding to the case $VI_{3}$. 
Since we have $Y^{\pm}_{0}(0,0) = 0$, $\frac{\partial Y^{+}_{0}}{\partial x}(0,0) > 0$ and $\frac{\partial Y^{-}_{0}}{\partial x}(0,0) < 0$, then we have a stable sliding region for $x < 0$ and an unstable sliding region for $x > 0$, at least for $x$ sufficiently close to zero. Thus items (1) and (3) are true. It remains to check item (2). Observe that, by item (1), $X^{sl}_{0}$ is well defined in $[\delta_{1},\delta_{2}]\setminus \{0\}$ and extends smoothly to $x=0$. Indeed, since $Y_0^+-Y_0^-$ has a simple zero at the origin, both the numerator and the denominator in \eqref{eq-sliding-x} have a common factor. We define 
\begin{equation*}
X^{sl}_{0}(0) := \displaystyle\lim_{x \rightarrow 0} \left(\displaystyle\frac{\operatorname{det}Z_{0}}{Y^{+}_{0} - Y^{-}_{0}}\right)(x,0) = \left(\displaystyle\frac{\frac{\partial Y^{+}_{0}}{\partial x}X^{-}_{0} - \frac{\partial Y^{-}_{0}}{\partial x}X^{+}_{0}}{\frac{\partial Y^{+}_{0}}{\partial x} - \frac{\partial Y^{-}_{0}}{\partial x}}\right)(0,0).
\end{equation*}

Since $\left(\frac{\partial Y^{+}_{0}}{\partial x}X^{-}_{0} - \frac{\partial Y^{-}_{0}}{\partial x}X^{+}_{0}\right)(0,0)>0$ in both cases, the proof of item (2) is complete.
\end{proof}

\begin{assumption}\label{assump-versal}
Throughout this paper, the two-fold $0\in\Sigma_{0}^{T}$ is $VV_{1}$ or $VI_{3}$, and the piecewise smooth vector field $Z_{\lambda}$ given in Equation \eqref{eq-pwl-smooth} with $h(x,y) = y$ satisfies the following assumption at $(x,y,\lambda) = (0,0,0)$:
\begin{equation}\label{eq-regularity}
\displaystyle\frac{\partial Y^{-}_{0}}{\partial \lambda}\displaystyle\frac{\partial Y^{+}_{0}}{\partial x} \neq \displaystyle\frac{\partial Y^{+}_{0}}{\partial \lambda}\displaystyle\frac{\partial Y^{-}_{0}}{\partial x}.
\end{equation}
\end{assumption}

The condition in \eqref{eq-regularity} played an important role in \cite{HK23,HK24} in proving the existence of limit cycles in regularizations of systems with a two-fold of type $VI_{3}$. Assumption \ref{assump-versal} can easily be satisfied by a suitable one-parameter unfolding. For example, one may take
\begin{equation*}
Y^{+}_{\lambda}(x,y) = Y^{+}_{0}(x,y)+\lambda,
\qquad
Y^{-}_{\lambda}(x,y) = Y^{-}_{0}(x,y).    
\end{equation*}
Since $\frac{\partial Y^{+}_{0}}{\partial x}(0,0) > 0$ and $\frac{\partial Y^{-}_{0}}{\partial x}(0,0) < 0$, such an unfolding satisfies \eqref{eq-regularity}. See also \cite{BLM18} for further details on the meaning of such condition.

We now define what kind of sliding cycles we are interested in. Assume that the Poincar\'e half-maps $\Pi^{a,b}_{0}$ of $Z^{\pm}_{0}$ are well defined in a compact interval $K\subset(0,\delta_{2}]$, and their ranges are contained in $[\delta_{1}, 0)$, where $\delta_1$ and $\delta_2$ are given by Proposition~\ref{prop-vv1}. Now, choose $x\in K$. The sliding cycle given by the union of the compact segment $[\Pi^{a}_{0}(x),x]\subset [\delta_{1}, \delta_{2}]$ and the arc of trajectory of $Z^{+}_{0}$ passing through $x$ will be denoted by $\Gamma^{a}_{x}$. Analogously, the sliding cycle given by the union of the compact segment $[\Pi^{b}_{0}(x),x]\subset [\delta_{1}, \delta_{2}]$ and an arc of trajectory of $Z^{-}_{0}$ passing through $x$ will be denoted by $\Gamma^{b}_{x}$, see Fig. \ref{fig-sliding-cycles}. It is important to remark that in the notation $\Gamma^{a,b}_{x}$, the letters $a$ and $b$ will be used to refer to sliding cycles from ``above'' and ``below'', respectively.

Suppose that Assumptions \ref{assump-reg1}-\ref{assump-versal} hold. Then the \textit{slow divergence integral} along the segment $[\Pi^{a,b}_{0}(x),x]\subset [\delta_{1}, \delta_{2}]$ is defined as
\begin{equation}\label{eq-sdi-general}
I^{a,b}(x) = \displaystyle\int_{\Pi^{a,b}_{0}(x)}^{x}\frac{(Y^{+}_{0} - Y^{-}_{0})^{2}}{\operatorname{det}Z_{0}}(s,0)\phi'\left(\phi^{-1}\left(\frac{-Y^{-}_{0}}{Y^{+}_{0} - Y^{-}_{0}}(s,0)\right)\right)ds,
\end{equation}
where $\det Z_{0}$ is given in Equation \eqref{eq-sliding-x}. The slow divergence integral is the integral of the divergence of the associated slow-fast system for $\varepsilon = 0$ (see \eqref{(3.4)} in Appendix \ref{sec-blow-ups}) whose integration variable is the time of the sliding vector field $X^{sl}_{0}$ (see \cite[Chapter 5]{DMDR} for further details on slow divergence integrals). In our setting, the slow divergence integral \eqref{eq-sdi-general} is independent of whether the origin is $VV_{1}$ or $VI_{3}$. On the other hand, the integral $I^{a}$ is related to sliding cycles from ``above'', whereas $I^{b}$ is related to sliding cycles from ``below''.

In the next definition, consider the vector field \eqref{eq-psvf-reg} with $h(x,y) = y$ and introduce the rescaling
\begin{equation*}
    \lambda = \varepsilon \widetilde{\lambda}, \qquad \widetilde{\lambda} \sim 0,
\end{equation*}
and then we deal with $Z_{\varepsilon,\varepsilon \widetilde{\lambda}}^{\phi}$.

\begin{definition}
Suppose $x_{0}\in K$. The \textit{cyclicity of the sliding cycle $\Gamma^{a,b}_{x_0}$} is bounded by $N \in \mathbb{N}$ if there exist $\varepsilon_0 > 0$, $\delta_{0} > 0$, and a neighborhood $\mathcal{U}$ of $0$ in the $\widetilde{\lambda}$-space such that $Z_{\varepsilon,\varepsilon \widetilde{\lambda}}^{\phi}$ has at most $N$ limit cycles, lying within Hausdorff distance $\delta_{0}$ of $\Gamma^{a,b}_{x_0}$, for all $(\varepsilon,\widetilde\lambda) \in (0,\varepsilon_0] \times \mathcal{U}$. We call the smallest $N$ with this property the \textit{cyclicity of $\Gamma^{a,b}_{x_0}$} and denote it by $\operatorname{Cycl}(\Gamma^{a,b}_{x_0})$. The cyclicity $\operatorname{Cycl}\left(\cup_{x\in K}\Gamma^{a,b}_{x}\right)$ is defined in a similar fashion.
\end{definition}

We have the following general result.

\begin{mtheorem}\label{theo-cycicity}
Consider the (rescaled) regularized vector field $Z_{\varepsilon,\varepsilon \widetilde{\lambda}}^{\phi}$ and suppose that Assumptions \ref{assump-reg1} to \ref{assump-versal} are satisfied. If $x_{0}\in K$, then the following statements are true. 
\begin{itemize}
    \item[(1)] If $I^{a}(x_0) < 0$ (resp.\ $I^{a}(x_0) > 0$), then $\operatorname{Cycl}(\Gamma^{a}_{x_0}) = 1$ and the limit cycle is hyperbolic and attracting (resp.\ repelling) when it exists. Moreover, if $I^{a}$ has no zeros in $K$, then
    \begin{equation*}
\operatorname{Cycl}\left(\displaystyle\bigcup_{x \in K} \Gamma^{a}_{x} \right) = 1.    
    \end{equation*}

    \item[(2)] If $I^{a}$ has a zero of multiplicity $l \geq 1$ at $x = x_{0}$, then
    $\operatorname{Cycl}(\Gamma^{a}_{x_{0}}) \leq l + 1$. When $I^{a}$ has at most $l \geq 1$ zeros in $K$ counting multiplicity, then 
    \begin{equation*}
\operatorname{Cycl}\left(\displaystyle\bigcup_{x \in K} \Gamma^{a}_{x} \right) \leq l + 1.    
    \end{equation*}

    \item[(3)] Suppose that $I^{a}$ has exactly $l \geq 1$ simple zeros $
    x_{1} < \cdots < x_{l}$ in the interior of $K$. If $x_{l+1}\in K$ and $x_{l+1} > x_{l}$, then there exists a smooth function $\widetilde{\lambda}_{c} = \widetilde{\lambda}_{c}(\varepsilon)$ with $ \widetilde{\lambda}_{c}(0) = 0$, such that \eqref{eq-psvf-reg} with $Z^{\phi}_{\varepsilon,\varepsilon\widetilde{\lambda}_{c}(\varepsilon)}$ has $l+1$ periodic orbits $
    \mathcal{O}_1^{\varepsilon}, \ldots, \mathcal{O}_{l+1}^{\varepsilon}$,
    for each $\varepsilon \sim 0$ and $\varepsilon > 0$. The periodic orbit $\mathcal{O}_{i}^{\varepsilon}$ is isolated, hyperbolic, and Hausdorff close to the sliding cycle $\Gamma^{a}_{x_{i}}$, for each $i = 1, \ldots, l+1$.
\end{itemize}
One has completely analogous statements concerning zeros of $I^{b}$ and cyclicity of sliding cycles from below $\Gamma^{b}_{x}$.
\end{mtheorem}

The proof of Theorem \ref{theo-cycicity} was addressed in \cite{HK23} for sliding cycles from below $\Gamma^{b}_{x}$ in the case $VI_{3}$ (that is, in the half-plane with the invisible fold point). In Appendix \ref{sec-blow-ups}, we prove Theorem \ref{theo-cycicity} for sliding cycles from above and below in the $VV_{1}$ case; and for the sliding cycles from above in the $VI_{3}$ case. The proof is similar to the proof of \cite[Theorem 3.1]{HK23} once we show that these new cases have similar properties after cylindrical blow-ups. We point out that Theorem \ref{theo-cycicity} is a general result, which can be applied not only to regularized PWL systems, but also to any regularized PWS system satisfying Assumptions \ref{assump-reg1} to \ref{assump-versal}. Finally, we remark that Theorem \ref{theo-cycicity} does not cover limit cycles that can be born from the origin $0\in\Sigma$. 

Next, our goal is to apply Theorem \ref{theo-cycicity} to regularized PWL systems. For
$\lambda = 0$, consider the PWL vector field
\begin{equation}
\label{pwl-system-added}
Z_0(x,y)=
\begin{cases}
Z_0^+(x,y)=
\bigl(
b^+ +a_{1,1}^+x+a_{1,2}^+y,\,
a_{2,1}^+x+a_{2,2}^+y
\bigr), & y>0,\\[1mm]
Z_0^-(x,y)=
\bigl(
b^- +a_{1,1}^-x+a_{1,2}^-y,\,
a_{2,1}^-x+a_{2,2}^-y
\bigr), & y<0,
\end{cases}
\end{equation}
where $b^+>0$, $a_{2,1}^+>0$, and $a_{2,1}^-<0$. Moreover,
$b^->0$ in the $VV_1$ case, whereas $b^-<0$ in the $VI_3$
case. Let $D_\pm$ and $T_\pm$ denote, respectively, the determinant
and the trace of the matrix associated with $Z_0^\pm$. To ensure
that the corresponding Poincar\'e half-maps are well defined, we
assume
\begin{equation}
D_\pm>0,
\qquad
4D_\pm-(T_\pm)^2>0\nonumber
\end{equation}
in the $VV_1$ case, and
\begin{equation}
D_+>0,
\qquad
4D_+-(T_+)^2>0\nonumber
\end{equation}
in the $VI_3$ case. For sliding cycles from above, the condition
with the plus sign is required, whereas for sliding cycles from
below, the condition with the minus sign is required whenever it
applies.



In this PWL setting, the domains of the Poincar\'e half-maps $\Pi^{a,b}_0$ are half-open intervals contained in $[0,+\infty)$. For each $\Pi^{a,b}_0$, let $\mathcal J^{a,b}$ denote the maximal half-open interval contained in its domain such that $X^{sl}_0$ does not vanish on $[\Pi^{a,b}_0(x),x]$ for every $x\in\mathcal J^{a,b}$. We call $\mathcal J^{a,b}$ the \textit{maximal domain} of the slow divergence integral $I^{a,b}$. Notice that $\mathcal J^{a,b}$ may be empty. Whenever $\mathcal J^{a,b}\neq\emptyset$, Theorem \ref{theo-cycicity} can be applied to any compact interval $K$ contained in the interior of $\mathcal J^{a,b}$. For more details on $\Pi_{0}^{a,b}$ and the domains $\mathcal J^{a,b}$, see Section \ref{sec-int-charac}.

We introduce the following quantities:
\begin{equation}\label{eq-def-coefficients}
  \begin{array}{rclcrcl}
    \Delta & = &a_{2,1}^{+}b^{-} - a_{2,1}^{-}b^{+}, & \quad & C & = & a_{2,1}^{+}a_{1,1}^{-} - a_{2,1}^{-}a_{1,1}^{+}, \\
    & & & & & & \\
    \xi_{0}^{a} & = & a_{2,1}^{+}b^{+}\displaystyle\frac{C}{\sqrt{\Delta}} - T_{+}\sqrt{\Delta}, & \quad & \xi_{\infty}^{a} & = & D_{+}\displaystyle\frac{C^{2}}{\Delta}, \\
     & & & & & & \\
    \xi_{0}^{b} & = & - a_{2,1}^{-}b^{-}\displaystyle\frac{C}{\sqrt{\Delta}} - T_{-}\sqrt{\Delta}, & \quad & \xi_{\infty}^{b} & = & -D_{-}\displaystyle\frac{C^{2}}{\Delta}.
  \end{array}
\end{equation}

In Section \ref{sec-pwl-twofold} we explain why the constant $\Delta$ is always positive under our assumptions (see also Equation \eqref{eq-slidingvf-condition}), so the constants $\xi_{0,\infty}^{a,b}$ are always well defined.

\begin{mtheorem}\label{thm-1}
Let $c\in\{a,b\}$. Suppose that each $Z_{\lambda}^{\pm}$ is affine with respect to $(x,y)$ and $C^\infty$-smooth with respect to $\lambda\sim0$,
that $Z_0$ is given by \eqref{pwl-system-added}, that the origin is a
two-fold of type $VV_1$ or $VI_3$, and that Assumptions \ref{assump-reg1}-\ref{assump-versal} and the corresponding conditions above are satisfied. Consider the slow divergence integrals $I^{c}$ on their maximal domains
$\mathcal{J}^{c}$. Then the following holds: 

If $\xi_{0}^{c}\neq 0$ or $\xi_{\infty}^{c}\neq 0$, then $I^{c}$ has at most one zero counting multiplicity in the interior of $\mathcal{J}^{c}$. If $\xi_{0}^{c} = \xi_{\infty}^{c} = 0$, then $I^{c}$ is identically zero on $\mathcal{J}^{c}$.
\end{mtheorem}

The proof of Theorem \ref{thm-1} is given in Section \ref{sec-zeros-sdi}. It follows an idea suggested by an anonymous referee of \cite{HK24} and uses a case-independent approach based on the integral characterization of Poincar\'e half-maps developed in \cite{CFS2021}. The result for sliding cycles from below in the $VI_3$ case was proved in \cite[Theorem~3.1]{HK24}.


The following result is an immediate consequence of Theorems~\ref{theo-cycicity} and~\ref{thm-1}.
\begin{mtheorem}\label{mhtm-c}
Let $c\in\{a,b\}$. Under the assumptions of Theorem \ref{thm-1}, if $I^{c}$ has at most one zero counting multiplicity in the interior of $\mathcal{J}^{c}$ and $K$ is a compact
interval contained in the interior of $\mathcal{J}^{c}$, then
\[
\operatorname{Cycl}\left(\bigcup_{x\in K}\Gamma_x^{c}\right)\leq 2.
\]
\end{mtheorem}

The periodic orbits of system $Z_{\varepsilon,\varepsilon \widetilde{\lambda}}^{\phi}$ in the case $I^{a,b} \equiv 0$ are not considered in the present paper. We leave this degenerate case for future work, since its analysis is expected to depend on the choice of regularization function $\phi$.

\begin{mtheorem} \label{mthm-existence-region}
Under the assumptions of Theorem \ref{thm-1}, the following statements hold.
\begin{enumerate}
\item In both the $VV_{1}$ and $VI_{3}$ cases, $I^{a}$ has exactly one zero counting multiplicity in the interior of its maximal domain $\mathcal{J}^a$ if and only if $CT_{+}<0$.
\item In the $VV_{1}$ case, $I^{b}$ has exactly one zero counting multiplicity in the interior of its maximal domain $\mathcal{J}^b$ if and only if $CT_{-}<0$.
\end{enumerate}
\end{mtheorem}

Sufficient conditions for the existence of simple zero of $I^{b}$ in the case $VI_{3}$ was addressed in \cite{HK24}.

Theorem \ref{mthm-existence-region} is proved in Section \ref{proof-existence region}. Suppose that one of the conditions in Theorem \ref{mthm-existence-region} holds. Then the corresponding slow divergence integral has a unique simple zero in the interior of $\mathcal{J}^{a,b}$. Hence, one may choose a compact interval $K$ contained in the interior of $\mathcal{J}^{a,b}$ such that the assumptions of Theorem \ref{theo-cycicity}(3) are satisfied. It follows that there exists a smooth function $\widetilde\lambda_c(\varepsilon)$, with $\widetilde\lambda_c(0)=0$, such that the regularized system $Z^\phi_{\varepsilon,\varepsilon\widetilde\lambda_c(\varepsilon)}$ has two hyperbolic limit cycles for all sufficiently small $\varepsilon>0$.

    \label{mhtm-existence-region} 


\section{Piecewise linear vector fields}\label{sec-pwl}

In this section, PWL vector fields are addressed. More precisely, in Equation \eqref{eq-pwl-smooth}, we assume that each $Z^{\pm}_{\lambda}$ is linear with respect to $(x,y)$ and $C^{\infty}$-smooth with respect to the parameter $\lambda$, with $\lambda \sim 0$. With these assumptions in mind, Equation \eqref{eq-pwl-smooth} for $\lambda = 0$ becomes
\begin{equation}\label{eq-pwl}
Z_{0}(x,y) = \left\{
  \begin{array}{rcrcl}
  Z^{+}_{0}(x,y) & = & (b_{1}^{+} + a_{1,1}^{+}x + a_{1,2}^{+}y, \ b_{2}^{+} + a_{2,1}^{+}x + a_{2,2}^{+}y), & \text{if} & y > 0,  \\
  Z^{-}_{0}(x,y) & = & (b_{1}^{-} + a_{1,1}^{-}x + a_{1,2}^{-}y, \ b_{2}^{-} + a_{2,1}^{-}x + a_{2,2}^{-}y), & \text{if} & y < 0,
  \end{array}
\right.    
\end{equation}
and the switching manifold is the straight line $\Sigma = \{y = 0\}$. One can also write the linear vector fields $Z^{\pm}_{0}$ in their matrix form by setting
\begin{equation}\label{eq-pwl-matrix}
A^{\pm} = \left(
  \begin{array}{cc}
    a_{1,1}^{\pm} & a_{1,2}^{\pm} \\
    a_{2,1}^{\pm} & a_{2,2}^{\pm} \\
  \end{array}
\right),
\qquad
B^{\pm} = \left(
  \begin{array}{c}
    b_{1}^{\pm}  \\
    b_{2}^{\pm} \\
  \end{array}
\right),
\end{equation}
and therefore $Z^{\pm}_{0}(x,y) = A^{\pm}(x,y)^{T} + B^{\pm}$.

\subsection{Useful results about PWL vector fields}\label{subsec:PWL vector fields}

Here, we state results that will be useful in our proofs in Section \ref{sec-zeros-sdi}. For this purpose, one must introduce some notation. The determinant and trace of $A^{\pm}$ will be denoted by $D_{\pm}$ and $T_{\pm}$, respectively. Following \cite{CFSN2023}, we further define the following constants:
\begin{equation}\label{eq-pwl-a+-}
\alpha_{-}:= a_{2,1}^{-}b_{1}^{-} - a_{1,1}^{-}b_{2}^{-}, \qquad  \alpha_{+}:= a_{2,1}^{+}b_{1}^{+} - a_{1,1}^{+}b_{2}^{+},
\end{equation}
\begin{equation}\label{eq-pwl-xi}
\xi_{0}:= \alpha_{+}T_{-} - \alpha_{-}T_{+}, \qquad \xi_{\infty}:= T_{-}^{2}D_{+} - T_{+}^{2}D_{-}, \qquad \beta:= a_{2,1}^{-}b_{2}^{+} - a_{2,1}^{+} b_{2}^{-}.
\end{equation}

Finally, we say that the PWL vector field \eqref{eq-pwl} \textit{satisfies hypothesis \eqref{eq-hyp-h}} if the following set of conditions holds:
\begin{equation}\label{eq-hyp-h}
\tag{H}
\begin{cases}
a_{2,1}^{-} a_{2,1}^{+} > 0;\\
\alpha_{-} \leq 0 \text{ and } 4D_{-} - T_{-}^{2} > 0,\ \text{or } \alpha_{-} > 0;\\
\alpha_{+} \geq 0 \text{ and } 4D_{+} - T_{+}^{2} > 0,\ \text{or } \alpha_{+} < 0.
\end{cases}    
\end{equation}

Geometrically, the first line of hypothesis \eqref{eq-hyp-h} is important for the existence of crossing periodic orbits. 
On the other hand, the second (resp. third) line is a necessary and sufficient condition such that the Poincar\'e half-map of $Z^{-}_{0}$ (resp. $Z^{+}_{0}$) with respect to the transversal section $\Sigma = \{y = 0\}$ exists (see \cite{CFSN2023}). 

\begin{remark}
The constants $\alpha_{\pm}, \xi_{0,\infty}$ and $\beta$ were defined in \cite{CFSN2023} and in that paper they have a slightly different expression than in Equations \eqref{eq-pwl-a+-} and \eqref{eq-pwl-xi}. This happens because in \cite{CFSN2023} the switching line is given by $\Sigma = \{x = 0\}$. Therefore, by performing the change of coordinates $(x,y)\mapsto (y,x)$, the switching line is $\Sigma = \{y = 0\}$ and then one precisely obtains $\alpha_{\pm}, \xi_{0,\infty}$ and $\beta$ as in Equations \eqref{eq-pwl-a+-} and \eqref{eq-pwl-xi}. The same holds for hypothesis \eqref{eq-hyp-h}. 
\end{remark}

Now we are able to state results that will play an important role in the proofs in Section \ref{sec-zeros-sdi}. The proofs of Theorems \ref{thm-annulus} and \ref{thm-crossing} (stated and proved in \cite{CFSN2023} and \cite{CFSN2023ii}, respectively) are based on the integral characterization of Poincar\'e half-maps given in \cite{CFS2021}.

\begin{theorem}[Theorem 1, \cite{CFSN2023}]\label{thm-annulus}
The PWL vector field \eqref{eq-pwl} has a crossing period annulus if, and only if, the hypothesis \eqref{eq-hyp-h} holds, $\operatorname{sgn}(T_{+}) = -\operatorname{sgn}(T_{-})$ and $\xi_{0} = \xi_{\infty} = \beta = 0$.
\end{theorem}

\begin{theorem}[Theorem A and Proposition 10, \cite{CFSN2023ii}]\label{thm-crossing}
If the PWL vector field \eqref{eq-pwl} does not have sliding region, then it has at most one crossing limit cycle. This limit cycle, if it exists, is hyperbolic, $T_-T_+<0$ and $\xi_{0} \neq 0$. Moreover, it is asymptotically stable (resp. unstable) provided $\xi_{0} < 0$ (resp. $\xi_{0} > 0$).
\end{theorem}

\subsection{Two-folds and sliding vector field}\label{sec-pwl-twofold}

In this paper, the goal is to study canard cycles of regularized piecewise linear vector fields near two-folds of type $VV_{1}$ and $VI_{3}$ (following the terminology of \cite{HK23, KRG03}). We assume that for $\lambda = 0$ the PWL vector field has the form \eqref{eq-pwl} and that the two-fold (either $VV_{1}$ or $VI_{3}$) is positioned at the origin. We further assume that the attracting sliding region is contained in the negative branch of the $x$-axis, whereas the repelling sliding region is contained in the positive branch of the $x$-axis. In this framework, near the two-fold the sliding vector field $Z^{sl}_{0}$ points from the left to the right in a neighborhood of the origin.

Using Proposition \ref{prop-vv1}, straightforward computations lead to the following conditions on the coefficients of the PWL vector field \eqref{eq-pwl} for $\lambda = 0$:
    \begin{equation}\label{eq-conditions-vv1}
    \text{If $0\in\Sigma$ is $VV_{1}$, then} \qquad b_{1}^{+}, a_{2,1}^{+} >  0, \qquad b_{2}^{+} = b_{2}^{-} = 0, \qquad a_{2,1}^{-} < 0 \text{ and } b_{1}^{-} > 0.  
    \end{equation}
    \begin{equation}\label{eq-conditions-vi3}
    \text{If $0\in\Sigma$ is $VI_{3}$, then} \qquad b_{1}^{+}, a_{2,1}^{+} >  0, \qquad b_{2}^{+} = b_{2}^{-} = 0, \qquad a_{2,1}^{-} < 0 \text{ and } b_{1}^{-} < 0.   
    \end{equation}

Therefore, due to conditions \eqref{eq-conditions-vv1} and \eqref{eq-conditions-vi3}, from now on the PWL vector field \eqref{eq-pwl} (for $\lambda = 0$) will be simply written as
\begin{equation}\label{eq-pwl-nf}
Z_{0}(x,y) = \left\{
  \begin{array}{rcrcl}
  Z^{+}_{0}(x,y) & = & (b^{+} + a_{1,1}^{+}x + a_{1,2}^{+}y, \ a_{2,1}^{+}x + a_{2,2}^{+}y), & \text{if} & y > 0,  \\
  Z^{-}_{0}(x,y) & = & (b^{-} + a_{1,1}^{-}x + a_{1,2}^{-}y, \ a_{2,1}^{-}x + a_{2,2}^{-}y), & \text{if} & y < 0,
  \end{array}
\right.    
\end{equation}
with $b^{+} > 0$, $a_{2,1}^{+} > 0$ and $a_{2,1}^{-} < 0$. Moreover, if $0\in\Sigma$ is $VV_{1}$ then $b^{-} > 0$, and if $0\in\Sigma$ is $VI_{3}$ then $b^{-} < 0$. One further denotes $Z^{\pm}_{0} = (X^{\pm}_{0},Y^{\pm}_{0})$. In addition, the matrix $B^{\pm}$ in Equation \eqref{eq-pwl-matrix} is simply written as
\begin{equation*}
B^{\pm} = \left(
  \begin{array}{c}
    b^{\pm}  \\
    0\\
  \end{array}
\right).
\end{equation*}

The equilibrium point $p^{\pm}_{0}$ of $Z^{\pm}_{0}$ (when it exists) is given by
\begin{equation*}
p^{\pm}_{0} = \left(-\displaystyle\frac{b^{\pm}a^{\pm}_{2,2}}{D_{\pm}}, \displaystyle\frac{b^{\pm}a^{\pm}_{2,1}}{D_{\pm}}\right),    
\end{equation*}
where $D_{\pm}$ is the determinant of the matrix $A^{\pm}$ given in \eqref{eq-pwl-matrix}. In the framework of $VV_{1}$ and $VI_{3}$ two-folds, in order to assure that a return mechanism exists (that is, the Poincar\'e half-maps $\Pi^{a,b}_{0}$ are well defined), together with Equations \eqref{eq-conditions-vv1} and \eqref{eq-conditions-vi3} one should further require the following conditions on the coefficients of the PWL vector field \eqref{eq-pwl} for $\lambda = 0$:
    \begin{equation}\label{eq-return-vv1}
    \text{If $0\in\Sigma$ is $VV_{1}$, then} \qquad D_{\pm} >  0, \text{ and } 4D_{\pm} - (T_{\pm})^{2} > 0.  
    \end{equation}
    \begin{equation}\label{eq-return-vi3}
    \text{If $0\in\Sigma$ is $VI_{3}$, then} \qquad D_{+} >  0, \text{ and } 4D_{+} - (T_{+})^{2} > 0.   
    \end{equation}

In Section \ref{sec-int-charac} we discuss the domain of the Poincar\'e half-maps $\Pi^{a,b}_{0}$.

Finally, observe that the sliding vector field $Z^{sl}_{0}:\Sigma^{sl}_{0}\rightarrow T\Sigma$ associated to \eqref{eq-pwl-nf} is given by $Z^{sl}_{0}(x,y) = (X^{sl}_{0}(x),0)$, with
\begin{equation}\label{eq-pwl-nf-slidingvf}
X^{sl}_{0}(x) = \displaystyle\frac{(a_{2,1}^{+}b^{-} - a_{2,1}^{-}b^{+}) + (a_{2,1}^{+}a_{1,1}^{-} - a_{2,1}^{-}a_{1,1}^{+})x}{a_{2,1}^{+} - a_{2,1}^{-}} = \displaystyle\frac{\Delta + Cx}{a_{2,1}^{+} - a_{2,1}^{-}},
\end{equation}
where $\Delta$ and $C$ were defined in Equation \eqref{eq-def-coefficients}.

For the sake of simplicity, we refer to $X^{sl}_{0}$ given in Equation \eqref{eq-pwl-nf-slidingvf} as the sliding vector field. Observe that the expression of \eqref{eq-pwl-nf-slidingvf} does not depend on whether $0\in\Sigma$ is $VV_{1}$ or $VI_{3}$. Moreover, it is always true that $(a_{2,1}^{+} - a_{2,1}^{-}) > 0$ due to conditions \eqref{eq-conditions-vv1} and \eqref{eq-conditions-vi3}. Since the sliding vector field, in a neighborhood of the origin, points from the attracting branch to the repelling branch of the sliding region $\Sigma^{sl}_{0}$, then
\begin{equation}\label{eq-slidingvf-condition}
\Delta = (a_{2,1}^{+}b^{-} - a_{2,1}^{-}b^{+}) > 0.
\end{equation}

Condition \eqref{eq-slidingvf-condition} is always true if the origin is $VV_{1}$ or $VI_3$. 
Finally, the sliding vector field has a singularity $x^{*}\in\Sigma$ if and only if $a_{2,1}^{+}a_{1,1}^{-} \neq a_{2,1}^{-}a_{1,1}^{+}$ (or, equivalently, $C\neq 0$). If such singularity exists, then it is hyperbolic and 
\begin{equation}\label{singularity-simple}
x^{*} = - \displaystyle\frac{a_{2,1}^{+}b^{-} - a_{2,1}^{-}b^{+}}{a_{2,1}^{+}a_{1,1}^{-} - a_{2,1}^{-}a_{1,1}^{+}} = -\displaystyle\frac{\Delta}{C}.
\end{equation}

The slow divergence integral \eqref{eq-sdi-general} in the PWL case becomes
\begin{equation}\label{eq-sdi-reg}
I^{a,b}(x) = (a_{2,1}^{+} - a_{2,1}^{-})\phi'\left(\phi^{-1}\left(\displaystyle\frac{-a_{2,1}^{-}}{a_{2,1}^{+} - a_{2,1}^{-}}\right)\right)\displaystyle\int_{\Pi^{a,b}_{0}(x)}^{x}\frac{udu}{X^{sl}_{0}(u)}. 
\end{equation}

\begin{remark}
In the PWL case, the factor involving $\phi$ in \eqref{eq-sdi-general} is a positive constant and can therefore be taken outside the integral. Hence, the zeros of $I^a$ and $I^b$ do not depend on the choice of the regularization function $\phi$. Outside the PWL setting, the situation is different. As shown in \cite{HK23}, we can produce arbitrarily many limit cycles in the $VI_3$ case from below by varying $\phi$, even when $Z^-_\lambda$ is linear and $Z^+_\lambda$ is quadratic.
\end{remark}

\subsection{Domains of the Poincar\'e half-maps and slow divergence integrals}\label{sec-int-charac} 

In this section we study the domains of $\Pi^{a,b}_0$ and $I^{a,b}$, given in \eqref{eq-sdi-reg}, and recall the integral characterization used in the proof of Theorem \ref{thm-1}. The domains of the Poincar\'e half-maps $\Pi^{a,b}_0$ are half-open intervals contained in $[0,+\infty)$.

Recall that, whenever the half-map is associated with a visible fold, conditions \eqref{eq-return-vv1}--\eqref{eq-return-vi3} give $D_\pm>0$ and $4D_\pm-(T_\pm)^2>0$. Thus, the corresponding linear vector field has a center or a focus in the half-plane containing the fold. This applies to both half-maps in the $VV_1$ case and to $\Pi^a_0$ in the $VI_3$ case. If the singularity is a center, the domain is $[0,+\infty)$ and the half-map takes values in $(-\infty,0]$. If it is an attracting focus, there exists $\hat x_0>0$ such that $\Pi^{a,b}_0(\hat x_0)=0$; the domain is $[\hat x_0,+\infty)$ and the half-map takes values in $(-\infty,0]$. If it is a repelling focus, there exists $\hat x_1<0$ such that $\Pi^{a,b}_0(0)=\hat x_1$; the domain is $[0,+\infty)$ and the half-map takes values in $(-\infty,\hat x_1]$ (see \cite{CFS2021}).

In the $VI_3$ case, $\Pi^b_0$ is associated with an invisible fold; its domain has the form $[0,b)$ and its range has the form $(\ell,0]$, where $b\in(0,+\infty]$ and $\ell\in[-\infty,0)$. We refer to \cite{CFS2021} and \cite{HK24} for further details.

We now recall the integral characterization used in Section~\ref{sec-zeros-sdi} for linear systems with an invisible fold. Suppose that, after a linear change of coordinates and, if necessary, a reversal of time, the system is written as $Z(x,y)=(Dy+\hat b,-x+Ty)$. Under the hypotheses of \cite[Theorem~8]{CFS2021}, for every $x$ in the domain of the corresponding forward or backward Poincar\'e half-map, $\Pi(x)$ is the unique point on the opposite side of the origin satisfying
\[
\int_{\Pi(x)}^x\frac{-u\,du}{Du^2+\hat bTu+\hat b^2}=0.
\]
The integral is taken in the component on which the denominator does not vanish. This characterization will be applied to the auxiliary systems introduced in the proof of Theorem~\ref{thm-1}.

We finally describe the maximal domains $\mathcal J^{a,b}$ of $I^{a,b}$. If $X^{{sl}}_0$ has no singularities, then $\mathcal J^{a,b}$ coincides with the domain of $\Pi^{a,b}_0$. Otherwise, let $x^*$ be its unique singularity, given by \eqref{singularity-simple}. If $x^*>0$, then $\mathcal J^{a,b}$ consists of the points $x$ in the domain of $\Pi^{a,b}_0$ satisfying $x<x^*$. If $x^*<0$, it consists of the points satisfying $\Pi^{a,b}_0(x)>x^*$. Consequently, $\mathcal J^{a,b}$ is empty only for an attracting focus with $0<x^*\leq\hat x_0$, or for a repelling focus with $\hat x_1\leq x^*<0$. In the center case and for $\Pi^b_0$ in the $VI_3$ case, $\mathcal J^{a,b}$ is always nonempty.

\section{Number of zeros of the slow divergence integral}\label{sec-zeros-sdi}

This section is devoted to study zeros of the slow divergence integral \eqref{eq-sdi-reg} for both cases $VV_{1}$ and $VI_{3}$. For this purpose, we will use the integral characterization of Poincar\'e half-maps and Theorems \ref{thm-annulus} and \ref{thm-crossing} presented in Section \ref{sec-pwl}. 

In Equation \eqref{eq-sdi-reg}, $\Pi^{a,b}_{0}$ denotes the Poincar\'e half-map of $Z^{\pm}_{0}$ for $\lambda = 0$, and therefore simple zeros of $I^{a}$ assure the existence of hyperbolic limit cycles Hausdorff close to sliding cycles from ``above'', whereas simple zeros of $I^{b}$ assure the existence of hyperbolic limit cycles Hausdorff close to sliding cycles from ``below''. It is important to remark that the expression of $I^{a,b}$ does not depend on whether $0\in\Sigma$ is $VV_{1}$ or $VI_{3}$. 

Since the constants
\begin{equation*}
 (a_{2,1}^{+} - a_{2,1}^{-}), \qquad \phi'\left(\phi^{-1}\left(\displaystyle\frac{-a_{2,1}^{-}}{a_{2,1}^{+} - a_{2,1}^{-}}\right)\right)
\end{equation*}
are always positive, then $\tilde{x}$ is zero of $I^{a,b}(x)$ if, and only if $\tilde{x}$ is zero of
\begin{equation}\label{eq-alt-sdi}
\widetilde{I}^{a,b}(x) = \displaystyle\int_{\Pi^{a,b}_{0}(x)}^{x}\frac{-udu}{(a_{2,1}^{+}b^{-} - a_{2,1}^{-}b^{+}) + (a_{2,1}^{+}a_{1,1}^{-} - a_{2,1}^{-}a_{1,1}^{+})u} .    
\end{equation}

Firstly, zeros of the integral $\widetilde{I}^{b}$ given in \eqref{eq-alt-sdi} will be considered.

\subsection{Proof of Theorem \ref{thm-1}}\label{sec-proof-thm-1}

In this first part of the proof, we deal with the integral $I^{b}$ as given in Equation \eqref{eq-sdi-reg}. In this part of the proof, the fold of $Z^{-}_{0}$ can be visible or invisible, that is, the proof does not depend whether $0\in\Sigma$ is $VV_{1}$ or $VI_{3}$. Recall the constants $\Delta$, $C$ and $\xi_{0,\infty}^{a,b}$ defined in Equation \eqref{eq-def-coefficients}. Firstly, consider the linear vector field
\begin{equation}\label{eq-auxiliary-vf}
\overline{Z}^{+}(x,y) = \left(\sqrt{\Delta}, \quad -x + \left(\displaystyle\frac{C}{\sqrt{\Delta}}\right)y    \right),
\end{equation}
and define the PWL vector field
\begin{equation}\label{eq-new-pwl-proof}
\overline{Z}(x,y) = \left\{
  \begin{array}{llc}
  \overline{Z}^{+}(x,y), & \text{if} & y > 0,  \\
  Z^{-}_{0}(x,y), & \text{if} & y < 0,
  \end{array}
\right.    
\end{equation}
where $\overline{Z}^{+}$ and $Z^{-}_{0}$ are given as in Equations \eqref{eq-auxiliary-vf} and \eqref{eq-pwl-nf}, respectively.


For the PWL vector field $\overline Z$ in
\eqref{eq-new-pwl-proof}, the coefficients of $x$ in the second
components of $Z_0^-$ and $\overline Z^+$ are $a_{2,1}^-$ and
$-1$, respectively. Since $a_{2,1}^-<0$, their product is positive.
Hence, $\Sigma$ contains no sliding region, and the nonsliding
hypothesis of Theorem~\ref{thm-crossing} is satisfied. Moreover, applying formulas \eqref{eq-pwl-a+-} in \eqref{eq-new-pwl-proof}, we obtain
\begin{equation*}
\alpha_{-} = a_{2,1}^-b^-,
\qquad
\alpha_{+} = - \sqrt{\Delta} < 0.    
\end{equation*}

If the two-fold is of type $VV_1$, then $\alpha_{-} < 0$ and
$4D_{-} - T_{-}^{2} > 0$, whereas in the $VI_3$ case we have
$\alpha_{-} > 0$. Thus, the second condition in \eqref{eq-hyp-h} holds in both cases, while the third one follows from $\alpha_{+} < 0$. Therefore, $\overline Z$ satisfies hypothesis \eqref{eq-hyp-h}, and both Poincar\'e half-maps are well defined.


Applying formulas \eqref{eq-pwl-xi} in $\overline{Z}$ given in \eqref{eq-new-pwl-proof}, it can be checked that $\beta=0$ and 
\begin{equation*}
\begin{array}{rclcl}
\xi_{0}^b & = & - \displaystyle\frac{a_{2,1}^{-}b^{-}(a_{2,1}^{+}a_{1,1}^{-}-a_{2,1}^{-}a_{1,1}^{+})}{\sqrt{a_{2,1}^{+}b^{-}-a_{2,1}^{-}b^{+}}}-T_{-}\sqrt{a_{2,1}^{+}b^{-}-a_{2,1}^{-}b^{+}} & = & - a_{2,1}^{-}b^{-}\displaystyle\frac{C}{\sqrt{\Delta}} - T_{-}\sqrt{\Delta}, \\
\xi_{\infty}^{b} & = & - \displaystyle\frac{(a_{2,1}^{+}a_{1,1}^{-}-a_{2,1}^{-}a_{1,1}^{+})^{2}D_{-}}{a_{2,1}^{+}b^{-}-a_{2,1}^{-}b^{+}} & = & -D_{-}\displaystyle\frac{C^{2}}{\Delta},
\end{array}
\end{equation*}
where $D_{-}$ and $T_{-}$ are respectively the determinant and the trace of the matrix $A^{-}$ associated to the linear vector field $Z^{-}_{0}$, and then we obtain the coefficients in Equation \eqref{eq-def-coefficients}.

We now show that the trace condition in
Theorem \ref{thm-annulus} is automatically satisfied whenever
$\xi_0^b=\xi_\infty^b=0$. If $D_{-} \neq 0$, then
$\xi_{\infty}^{b} = 0$ implies $C=0$, and hence $\overline{T}_{+} = \frac{C}{\sqrt{\Delta}} = 0$, where $\overline{T}_{+}$ is the trace of the matrix associated with $\overline{Z}^{+}$. The equality $\xi_0^b=0$ then gives $T_-=0$. Therefore, $\operatorname{sgn}(\overline T_+)=-\operatorname{sgn}(T_-)$. On the other hand, if $D_-=0$, then we are necessarily in the $VI_3$ case, since
$D_->0$ in the $VV_1$ case (see \eqref{eq-return-vv1}). The equality $\xi_0^b=0$ gives
\begin{equation*}
T_-=-\frac{a_{2,1}^-b^-}{\sqrt{\Delta}}\,\overline T_+.    
\end{equation*}

In the $VI_3$ case, $a_{2,1}^-<0$ and $b^-<0$, so
$a_{2,1}^-b^->0$. It follows again that
$\operatorname{sgn}(\overline T_+)=-\operatorname{sgn}(T_-)$,
where $\operatorname{sgn}(0)=0$.

Consequently, the conclusions of
Theorems \ref{thm-annulus} and \ref{thm-crossing} reduce, for the
PWL vector field $\overline Z$, to the following two cases. If
$\xi_0^b=\xi_\infty^b=0$, then the trace condition required in
Theorem \ref{thm-annulus} is automatically satisfied, and
$\overline Z$ has a crossing period annulus. If
$\xi_0^b\neq0$ or $\xi_\infty^b\neq0$, then
Theorem \ref{thm-crossing} implies that $\overline Z$ has at most
one crossing limit cycle, which is hyperbolic whenever it exists.


Following Section \ref{sec-int-charac}, the backward Poincar\'e half-map $\overline{\Pi}^{a}$ of $\overline{Z}^{+}$ with respect to the transversal section $\Sigma = \{y = 0\}$ is then given by its integral characterization (see \cite[Theorem 8]{CFS2021})
\begin{equation}\label{eq-int-charac}
\displaystyle\int_{\overline{\Pi}^{a}(x)}^{x}\frac{-udu}{\Delta + Cu} = 0.    
\end{equation}

Assume first that $\xi_0^b\neq0$ or $\xi_\infty^b\neq0$. By the
integral characterization \eqref{eq-int-charac}, the slow divergence
integral $I^b$ in \eqref{eq-sdi-reg}, or equivalently
$\widetilde I^b$ in \eqref{eq-alt-sdi}, vanishes at $x$ if and only
if the sewing PWL vector field $\overline Z$ in
\eqref{eq-new-pwl-proof} has a crossing periodic orbit through
$(x,0)$. Theorem~\ref{thm-crossing} implies that $\overline Z$ has
at most one crossing limit cycle, which is hyperbolic whenever it
exists. Consequently, $I^b$ has at most one zero counting multiplicity. 





Assume now that $\xi_0^b=\xi_\infty^b=0$. As shown above, the trace
condition
$\operatorname{sgn}(\overline T_+)=-\operatorname{sgn}(T_-)$
is then automatically satisfied. Theorem~\ref{thm-annulus} therefore
implies that $\overline Z$ has a crossing period annulus. By
\eqref{eq-int-charac}, $\widetilde I^b$, and hence also $I^b$,
vanishes identically.


Now we proceed to the second part of the proof, which concerns the integral $I^{a}$ as given in Equation \eqref{eq-sdi-reg}. This proof is similar to the proof of the first part, and we will only highlight the steps for the sake of completeness. Consider the linear vector field
\begin{equation}\label{eq-auxiliary-vf-2}
\overline{Z}^{-}(x,y) = \left(\sqrt{\Delta}, \quad x + \left(\displaystyle\frac{C}{\sqrt{\Delta}}\right)y    \right),
\end{equation}
and define the PWL vector field
\begin{equation}\label{eq-new-pwl-proof-2}
\overline{Z}(x,y) = \left\{
  \begin{array}{llc}
  Z^{+}_{0}(x,y), & \text{if} & y > 0,  \\
  \overline{Z}^{-}(x,y), & \text{if} & y < 0,
  \end{array}
\right.    
\end{equation}
where $Z^{+}_{0}$ and $\overline{Z}^{-}$ are given as in Equations \eqref{eq-pwl-nf} and \eqref{eq-auxiliary-vf-2}, respectively. As in the first part of the proof, direct computations show that the PWL vector field $\overline Z$ defined in \eqref{eq-new-pwl-proof-2} satisfies hypothesis \eqref{eq-hyp-h} and is nonsliding. Indeed, along $y=0$, the normal components are $a_{2,1}^{+}x$ and $x$, which have the same sign for $x\neq0$. Moreover, applying formulas \eqref{eq-pwl-a+-} in \eqref{eq-new-pwl-proof-2}, we obtain $\alpha_{-} = \sqrt{\Delta} > 0$, whereas $\alpha_{+} = a_{2,1}^{+}b^{+}>0$ and $4D_{+}-T_{+}^{2}>0$, since $Z_{0}^{+}$ has a monodromic singularity. Hence, the corresponding Poincaré half-maps are well defined.

Concerning the PWL vector field $\overline{Z}$ given in \eqref{eq-new-pwl-proof-2}, applying formulas \eqref{eq-pwl-xi} it can be checked that $\beta=0$ and
\begin{equation*}
\begin{array}{rclcl}
\xi_{0}^a & = & \displaystyle\frac{a_{2,1}^{+}b^{+}(a_{2,1}^{+}a_{1,1}^{-}-a_{2,1}^{-}a_{1,1}^{+})}{\sqrt{a_{2,1}^{+}b^{-}-a_{2,1}^{-}b^{+}}} -T_{+}\sqrt{a_{2,1}^{+}b^{-}-a_{2,1}^{-}b^{+}} & = & a_{2,1}^{+}b^{+}\displaystyle\frac{C}{\sqrt{\Delta}} - T_{+}\sqrt{\Delta}, \\
\xi_{\infty}^a & = & \displaystyle\frac{(a_{2,1}^{+}a_{1,1}^{-}-a_{2,1}^{-}a_{1,1}^{+})^{2}D_{+}}{a_{2,1}^{+}b^{-}-a_{2,1}^{-}b^{+}} & = & D_{+}\displaystyle\frac{C^{2}}{\Delta},
\end{array}
\end{equation*}
where $D_{+}$ and $T_{+}$ are respectively the determinant and the trace of the matrix $A^{+}$ associated to the linear vector field $Z^{+}$. We obtained the coefficients defined in \eqref{eq-def-coefficients} and the rest of the proof follows exactly as in the proof of the first part.

\subsection{Proof of Theorem \ref{mthm-existence-region}}\label{proof-existence region}
We prove both statements simultaneously. For simplicity, let $\Pi$ denote the Poincar\'e half-map associated with the vector field containing the visible fold, and let $T$ and $D$ denote the trace and determinant of its matrix, respectively. Thus, $(\Pi,T,D)=(\Pi_{0}^{a},T_{+},D_{+})$ for sliding cycles from above, and $(\Pi,T,D)=(\Pi_{0}^{b},T_{-},D_{-})$ for sliding cycles from below in the $VV_{1}$ case.  In both cases, $D>0$ and $4D-T^{2}>0$.

Assume that $CT<0$. Then both $C$ and $T$ are nonzero, and the equilibrium point of the vector field containing the visible fold is a focus. Moreover, it lies in the half-plane containing the visible fold and the domain of the corresponding Poincar\'e half-map is described in Section \ref{sec-int-charac}. The focus is repelling when $T>0$ and attracting when $T<0$. It is enough to study the integral
\begin{equation*}
\widetilde I(x)=\int_{\Pi(x)}^{x}\frac{u}{\Delta + Cu}\,du,    
\end{equation*}
which has the same zeros as the corresponding slow divergence integral. We distinguish the two possibilities allowed by $CT < 0$.

Suppose first that $C<0<T$. The focus is repelling, and hence there exists $\hat{x}_{1}<0$ such that $\Pi(0)=\hat{x}_{1}$, and the domain of $\Pi$ is $[0,+\infty)$ and the range of $\Pi$ is $(-\infty,\hat{x}_{1}]$. We have $x^{*} = -\frac{\Delta}{C} > 0$ (see \eqref{singularity-simple}). The maximal domain of $\widetilde I$ is $[0,x^{*})$, and
\begin{equation*}
\widetilde I(0)=\int_{\hat{x}_{1}}^{0}\frac{u}{\Delta+Cu}\,du < 0.
\end{equation*}

A primitive of the integrand of $\widetilde{I}$ is
\begin{equation*}
F(u)=\frac{u}{C}-\frac{\Delta}{C^{2}}\log(\Delta+Cu).
\end{equation*}

Since $\Pi(x)$ remains finite as $x\to(x^{*})^{-}$, whereas $\Delta + Cx\to 0^{+}$, we obtain
\begin{equation*}
\lim_{x\to(x^{*})^{-}} \widetilde I(x)
= \lim_{x\to(x^{*})^{-}}\left(F(x)-F\left(\Pi(x)\right)\right)
= +\infty.
\end{equation*}

Therefore, $\widetilde I$, and hence the corresponding slow divergence integral, has at least one zero in $(0,x^{*})$. 

Suppose now that $T<0<C$. The focus is attracting, and hence there exists $\hat{x}_{0}>0$ such that $\Pi(\hat{x}_{0})=0$, and the domain and range of $\Pi$ are given by $[\hat{x}_{0},+\infty)$ and $(-\infty,0]$, respectively. In this case, $x^{*}= - \frac{\Delta}{C}< 0$. Let $\overline{x} := \Pi^{-1}(x^{*})>\hat{x}_{0}$. The maximal domain of $\widetilde I$ is $[\hat{x}_{0},\overline{x})$, and
\begin{equation*}
\widetilde I(\hat{x}_{0})
=\int_{0}^{\hat{x}_{0}}\frac{u}{\Delta+Cu}\,du > 0.
\end{equation*}

As $x\to\overline{x}^{-}$, we have $\Pi(x)\to(x^{*})^{+}$, and therefore
\begin{equation*}
\lim_{x\to\overline{x}^{-}}\widetilde I(x)
=\lim_{x\to\overline{x}^{-}}\bigl(F(x)-F(\Pi(x))\bigr)
=-\infty.
\end{equation*}

Thus, the corresponding slow divergence integral again has at least one zero in the interior of its maximal domain.

By hypothesis, $C\neq0$ and $\Delta,D_{\pm}>0$, which implies $\xi_{\infty}^{a} = \frac{C^{2}D_{+}}{\Delta} \neq 0$ for sliding cycles from above, whereas $\xi_{\infty}^{b} = -\frac{C^{2}D_{-}}{\Delta} \neq 0$ for sliding cycles from below. By Theorem \ref{thm-1}, the corresponding slow divergence integral has at most one zero counting multiplicity. Consequently, the zero found above is unique and simple.

It remains to prove necessity. Suppose that the corresponding slow
divergence integral has a simple zero. As shown in the proof of
Theorem \ref{thm-1}, this zero corresponds to a crossing limit cycle of the
associated sewing PWL system (see \eqref{eq-new-pwl-proof} for $I^{b}$ and \eqref{eq-new-pwl-proof-2} for $I^{a}$). By Theorem \ref{thm-crossing}, the
traces of the two linear vector fields defining this sewing system must
have opposite signs. Since these traces are $T$ and $\frac{C}{\sqrt{\Delta}}$, respectively, and $\Delta>0$, we obtain
\begin{equation*}
\frac{CT}{\sqrt{\Delta}} < 0,    
\end{equation*}
which implies $CT<0$. This concludes the proof.

\begin{remark}
We remark that we cannot directly apply \cite[Corollary 2]{CFSN2023ii} to either \eqref{eq-new-pwl-proof} or \eqref{eq-new-pwl-proof-2} in order to prove that the condition $CT_{\pm} < 0$ in Theorem \ref{mthm-existence-region} is sufficient, because such PWL systems do not satisfy the hypotheses of \cite[Proposition 9]{CFSN2023ii}. Indeed, in both Equations \eqref{eq-new-pwl-proof} and \eqref{eq-new-pwl-proof-2} the determinant of the matrix associated to $\overline{Z}^{\pm}$ is identically zero. 
\end{remark}

\section{Cyclicity of sliding cycles in the presence of singularities of the sliding vector field}\label{sec-cycl-sing}

So far, we assumed that the sliding vector field $X^{sl}_{0}$ does not have singularities in a compact segment $[\delta_{1}, \delta_{2}]$ with $\delta_{1} < 0 < \delta_{2}$ and, in addition, it holds $X^{sl}_{0}(x,0) > 0$, for all $x\in[\delta_{1},\delta_{2}]$ (see Proposition \ref{prop-vv1}). In this section we discuss the cyclicity of sliding cycles in the case where either $\delta_{1}$ or $\delta_{2}$ is a singularity of $X^{sl}_{0}$.

More precisely, we consider the PWL vector field $Z_{0}$ in Equation \eqref{eq-pwl-nf} where $b^{+} > 0$, $a_{2,1}^{+} > 0$ and $a_{2,1}^{-} < 0$. When $0\in\Sigma$ is $VV_{1}$, then $b^{-} > 0$, and if $0\in\Sigma$ is $VI_{3}$, then $b^{-} < 0$. The associated sliding vector field is given by \eqref{eq-pwl-nf-slidingvf} and condition \eqref{eq-slidingvf-condition} is fulfilled. Recall from Section~\ref{sec-pwl-twofold} that the sliding vector field has a singularity if and only if
$a_{2,1}^{+}a_{1,1}^{-} \neq a_{2,1}^{-}a_{1,1}^{+}$,
in which case the singularity is denoted by $x^{*}$ and given by \eqref{singularity-simple}. In this section, we assume that $a_{2,1}^{+}a_{1,1}^{-} \neq a_{2,1}^{-}a_{1,1}^{+}$ (or, equivalently, $C\neq 0$).

If $x^{*}>0$, so that $x^{*}$ lies in the repelling sliding region, and 
$\Gamma^{a,b}_{x^{*}}$ is well defined, we consider this sliding cycle. On the other hand, if $x^{*} < 0$, so that $x^{*}$ lies in the attracting sliding region, assume that $x^{*}$ is in the range of $\Pi_0^{a,b}$ 
and consider $\Gamma^{a,b}_{(\Pi^{a,b}_{0})^{-1}(x^*)}$. In both cases, the corresponding sliding cycle has the singularity $x^*$ as its corner point, see Fig. \ref{fig:centerandfocus}.

\begin{mtheorem}\label{mthm-singularity}
Consider the regularization $Z_{\varepsilon,\varepsilon \widetilde{\lambda}}^{\phi}$ of a PWL vector field $Z_\lambda$, assuming that $Z_0$ is given by \eqref{eq-pwl-nf} and satisfies the properties stated above. Suppose further that Assumptions \ref{assump-reg1}-\ref{assump-reg3} hold.  Then
\begin{equation*}
 \operatorname{Cycl}\left(\Gamma^{a,b}_{x^{*}}\right) \leq 1 \text{ for } x^{*}>0, \qquad \text{ and } \qquad \operatorname{Cycl}\left(\Gamma^{a,b}_{(\Pi^{a,b}_{0})^{-1}(x^{*})}\right) \leq 1 \text { for } x^{*} < 0.   
\end{equation*}
Moreover, if such limit cycle exists, then it is hyperbolic and attracting (resp. repelling) if $x^{*} < 0$ (resp. $x^{*} > 0$).
\end{mtheorem}

Theorem \ref{mthm-singularity} was proved in \cite[Theorem~4.1]{HHPY26} for the sliding cycles ``from below'' $\Gamma^{b}_{x^{*}}$ and $\Gamma^{b}_{(\Pi^{b}_{0})^{-1}(x^{*})}$ in the $VI_3$ case, that is, in the half-plane containing the invisible fold. As in the case of Theorem~\ref{theo-cycicity}, Appendix \ref{sec-blow-ups} treats the remaining cases: sliding cycles from above and from below in the $VV_1$ case, and sliding cycles from above in the $VI_3$ case. The proof is similar to that of \cite[Theorem 4.1]{HHPY26}.

For Theorem \ref{mthm-singularity}, condition \eqref{eq-regularity} is not required. We stress that the proof of Theorem \ref{mthm-singularity} cannot be based on the slow divergence integral \eqref{eq-sdi-reg}, since this integral is not well defined near $x=x^*$ or near $x=(\Pi^{a,b}_{0})^{-1}(x^*)$. For this reason, one has to use instead the full divergence integral of $Z_{\varepsilon,\varepsilon \widetilde{\lambda}}^{\phi}$ introduced in \cite{HHPY26}. Its sign is determined by the singularity at $x=x^*$: the integral is negative when $x^*<0$, and positive when $x^*>0$.


Fig. \ref{fig:centerandfocus} deals with sliding cycles in the presence of singularity of $X^{sl}_{0}$ in the $VV_{1}$ case. This framework is possible if the PWL vector field \eqref{eq-pwl-nf} satisfies conditions \eqref{eq-conditions-vv1} and \eqref{eq-return-vv1}. As discussed before, the singularity $x^{*}$ of $X^{sl}_{0}$ exists if, and only if $a_{2,1}^{+}a_{1,1}^{-} \neq a_{2,1}^{-}a_{1,1}^{+}$ (or $C\neq 0$). When all of those conditions are satisfied, $x^{*}$ is not positioned in the origin and then the sliding cycles $\Gamma^{a,b}_{x^{*}}$ or $\Gamma^{a,b}_{(\Pi^{a,b}_{0})^{-1}(x^{*})}$ are possible. 

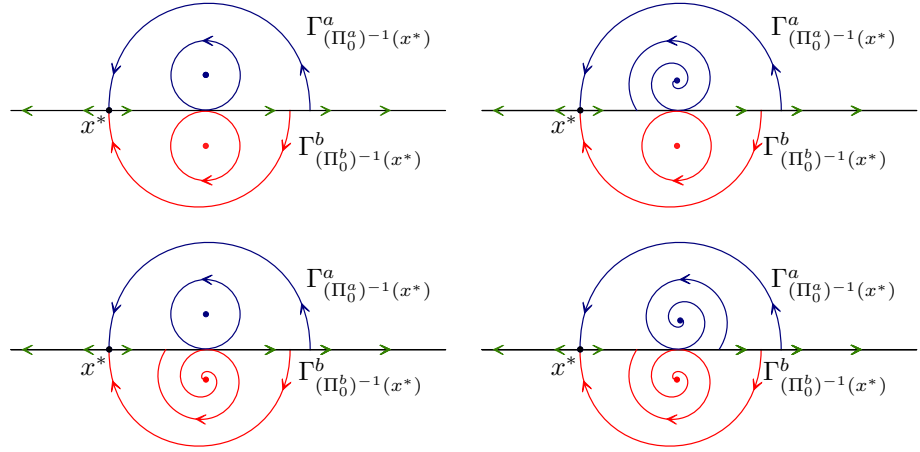
\begin{figure}[htb]
	\begin{center}
	\includegraphics[width=12cm]{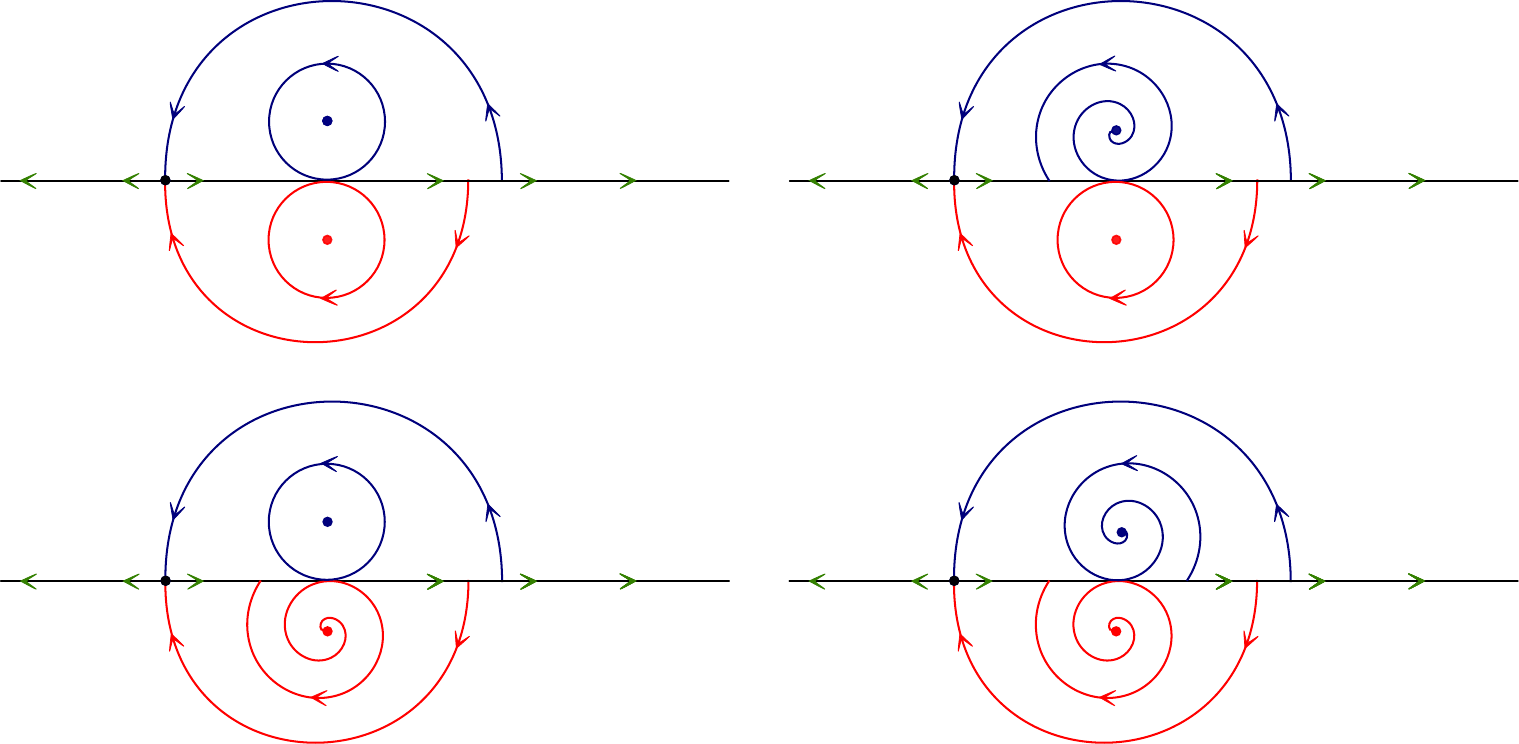}
       {\footnotesize 
       \put(-138,27){$x^{*}$}
       \put(-138,118){$x^{*}$}
       \put(-315,27){$x^{*}$}
       \put(-315,118){$x^{*}$}
       \put(-233,110){$\Gamma^{b}_{(\Pi^{b}_{0})^{-1}(x^{*})}$}
        \put(-230,155){$\Gamma^{a}_{(\Pi^{a}_{0})^{-1}(x^{*})}$}
        \put(-233,25){$\Gamma^{b}_{(\Pi^{b}_{0})^{-1}(x^{*})}$}
        \put(-230,60){$\Gamma^{a}_{(\Pi^{a}_{0})^{-1}(x^{*})}$}
     \put(-58,110){$\Gamma^{b}_{(\Pi^{b}_{0})^{-1}(x^{*})}$}
        \put(-55,155){$\Gamma^{a}_{(\Pi^{a}_{0})^{-1}(x^{*})}$}
        \put(-58,25){$\Gamma^{b}_{(\Pi^{b}_{0})^{-1}(x^{*})}$}
        \put(-55,60){$\Gamma^{a}_{(\Pi^{a}_{0})^{-1}(x^{*})}$}}
         \end{center}
    \caption{Examples of sliding cycles $\Gamma^{a,b}_{(\Pi^{a,b}_{0})^{-1}(x^{*})}$ in the presence of the singularity $x^{*}$ as considered in Theorem \ref{mthm-singularity}. In all four cases, the two-fold $VV_{1}$ is sketched and $x^{*} = \delta_{1} < 0$.}
    \label{fig:centerandfocus}
\end{figure}

\section*{Declarations}
\textbf{Conflict of interest} The authors declare that they have no conflict of interest.

\textbf{Data Availability Statement} Data sharing not applicable to this article as no datasets were generated or analyzed during the current study.

\section*{Acknowledgments and funding}

R. Huzak is supported by Croatian Science Foundation (HRZZ) grant IP-2022-10-9820. K. U. Kristiansen is supported by Danish Research Council (DFF) Grant 4283-00014B. O. H. Perez is supported by Sao Paulo Research Foundation (FAPESP) grant 2021/10198-9. A. L. M. T. da Silva and R. Huzak are supported by the Special Research Fund (BOF number: BOF25BL05) of Hasselt University. 

\appendix

\section{Proofs of Theorems \ref{theo-cycicity} and \ref{mthm-singularity}}\label{sec-blow-ups}

Throughout Appendix \ref{sec-blow-ups}, the two-fold is assumed to be type $VV_1$ or $VI_3$. For Theorem \ref{theo-cycicity} we assume Assumptions \ref{assump-reg1}-\ref{assump-versal}, whereas for Theorem \ref{mthm-singularity} only \ref{assump-reg1}-\ref{assump-reg3} are needed.

The proof is divided into several steps. First, we perform a suitable cylindrical blow-up, which we call the primary blow-up, in the regularized system $Z_{\varepsilon,\varepsilon \widetilde{\lambda}}^{\phi}$ augmented by the trivial equation $\dot\varepsilon=0$. In this way, the discontinuity line $\{y=0\}$ is replaced by a half-cylinder. Near the canard trajectories lying on the top of this cylinder, we verify assumptions T0--T6 of \cite{DD05} in the case of Theorem~\ref{theo-cycicity}, and the assumptions of \cite{DD08} in the case of Theorem~\ref{mthm-singularity}. A secondary cylindrical blow-up plays an important role in this verification. Once these assumptions are satisfied, one can study the transition maps and the structure of the difference map of $Z_{\varepsilon,\varepsilon \widetilde{\lambda}}^{\phi}$ near the sliding cycles. One then relates the zeros of this difference map to the zeros of the slow divergence integral \eqref{eq-sdi-general} in the case of Theorem~\ref{theo-cycicity}, and to the zeros of the full divergence integral of $Z_{\varepsilon,\varepsilon \widetilde{\lambda}}^{\phi}$ in the case of Theorem~\ref{mthm-singularity}. The analysis of the difference map follows the same arguments as in \cite{HK23,HHPY26}, and is therefore omitted for the sake of readability.

\subsection{Primary cylindrical blow-up}
Recall from Section~\ref{sec-main-cyclicity} that $Z_{\varepsilon,\varepsilon \widetilde{\lambda}}^{\phi}$ denotes the vector field \eqref{eq-psvf-reg} with $h(x,y)=y$, after the rescaling
\begin{equation*}
    \lambda = \varepsilon \widetilde{\lambda}, \qquad \widetilde{\lambda} \sim 0.
\end{equation*} The parameter $\widetilde{\lambda}$ will be called \textit{regular breaking parameter}. In order to study the dynamics of $Z_{\varepsilon,\varepsilon \widetilde{\lambda}}^{\phi}$ near $\Sigma = \{y = 0\}$, we consider the extended vector field $Z_{\varepsilon,\varepsilon \widetilde{\lambda}}^{\phi} + 0\frac{\partial}{\partial \varepsilon}$ and perform the cylindrical blow up
\begin{equation}\label{(3.2)}
\begin{array}{rccc}
\Phi: & \mathcal{M} & \longrightarrow &\mathbb{R}^{3} \\
& (x,\overline{y},\overline{\varepsilon},r) & \longmapsto & (x,r^{2}\overline{y}, r\overline {\varepsilon})
= (x,y,\varepsilon),
\end{array}
\end{equation}
with $\mathcal{M}$ being a manifold with corners,
$(\overline{y},\overline{\varepsilon})\in \mathbb{S}^{1}$,
$x\in\mathbb{R}$ and $\overline{\varepsilon}, r\geq0$.


Let $\overline{G} := \Phi^{*}
\left( Z_{\varepsilon,\varepsilon \widetilde{\lambda}}^{\phi} + 0\frac{\partial}{\partial \varepsilon} \right)$ be the vector field on $\mathcal{M}$, which is the pullback of $Z_{\varepsilon,\varepsilon \widetilde{\lambda}}^{\phi} + 0\frac{\partial}{\partial \varepsilon}$ under \eqref{(3.2)}. We will specifically study $\widehat{G}\colon = r^{2}\overline{G}$. To analyze the dynamics of $\widehat{G}$ in a neighborhood of the cylinder, we consider different charts.

\begin{remark} From now on, we adopt the notation $Z^{\pm}_{\varepsilon \widetilde{\lambda}}(x,y) := Z^{\pm}(x,y,\varepsilon \widetilde{\lambda})$, $X^{\pm}_{\varepsilon \widetilde{\lambda}}(x,y) := X^{\pm}(x,y,\varepsilon \widetilde{\lambda})$ and $Y^{\pm}_{\varepsilon \widetilde{\lambda}}(x,y) := Y^{\pm}(x,y,\varepsilon \widetilde{\lambda})$.
\end{remark}

\subsubsection{Dynamics in the chart $\overline{\varepsilon} = 1$} 

Using the chart-specific coordinate $y_2$ defined by $y=\varepsilon^2y_2$, with $(x,y_2)$ kept in a large compact subset of $\mathbb{R}^2, \varepsilon > 0$ and $\varepsilon \sim 0$, we obtain 
\begin{equation}\label{eq-psvf-reg-appendix}
\begin{array}{rcl}
		\dot{x} & = & \varepsilon^2\left(
X^{+}(x,\varepsilon^2y_2,\varepsilon\widetilde{\lambda})\phi(y_2) +X^{-}(x,\varepsilon^{2}y_{2},\varepsilon\widetilde{\lambda})(1 -\phi(y_2))\right),\\
\dot{y_2} & = & Y^{+}(x,\varepsilon^2y_2,\varepsilon\widetilde{\lambda})\phi(y_2)+ Y^{-}(x,\varepsilon^2y_2,\varepsilon\widetilde{\lambda})(1-\phi(y_2)).
\end{array}
\end{equation}

Setting $\varepsilon=0$ in system \eqref{eq-psvf-reg-appendix}, one obtains
\begin{equation}\label{(3.4)}
		\begin{array}{rcl}
			\dot{x} & = & 0,\\
			\dot{y_2} & = & Y^{+}(x,0,0)\phi(y_2)+ Y^{-}(x,0,0)(1-\phi(y_2)).
		\end{array}
\end{equation}

The critical set of \eqref{(3.4)} is given by the union of critical manifolds
\begin{equation*}
	H := \{(0, y_2); \ y_2\in\mathbb{R}\}, \quad C := \left\{(x, y_{2}); \  y_2 =\phi^{-1}\left( \frac{-Y^{-}}{Y^{+}-Y^{-}}
(x,0,0)\right)\right\}.
\end{equation*}

Note that an intersection of ${H}$ and ${C}$ occurs at the point
\begin{equation*}
	p_0 =\left(0,\phi^{-1}\left( \frac{-\frac{\partial Y^{-}}{\partial x}}{\frac{\partial Y^{+}}{\partial x}-\frac{\partial Y^{-}}{\partial x}}
	(0,0,0)\right)\right).
\end{equation*}

All singularities on $H$ are nilpotent except for $p_0$, whose linearization is identically zero.

In this subsection we show that the slow-fast system \eqref{eq-psvf-reg-appendix} satisfies Conditions T0–T2 of \cite{DD05} along the critical curve $C$. The singularities on $C$ are normally attracting for $x<0$ and normally repelling for $x>0$. Moreover, by Proposition~\ref{prop-vv1}, the slow dynamics $X^{sl}(x)$ given in \eqref{eq-sliding-x} is regular and points from the attracting part of $C$ to its repelling part; see Fig.~\ref{FirstBlowup}. Therefore, denoting the vector field in \eqref{eq-psvf-reg-appendix} by $\widehat{G}_S$, and letting $M_x$ be any local $C^n$ center manifold of $\widehat{G}^1_S:=\widehat{G}_S+0\frac{\partial}{\partial \varepsilon}$ at a normally hyperbolic singularity $x\in C$, the rescaled vector field $\frac{1}{\varepsilon^2}\widehat{G}_S^1|_{M_x}$ defines a local flow box containing $C$ and oriented from left to right. (We call the exponent $2$ in the term $\varepsilon^2$ the order of degeneracy.) This implies that Assumptions T0–T2 of \cite{DD05} are satisfied. 

Notice that the assumptions of Theorem~\ref{mthm-singularity} imply that $X^{sl}(x)$ has a singularity away from $x=0$, which puts us in the framework of \cite{DD08}.

\subsubsection{Dynamics in the phase directional charts $\overline{y}=\pm 1$}

In the chart $\overline{y}=-1$ associated with \eqref{(3.2)}, we use the  coordinates $(r_1,\varepsilon_1)$ given by $(y,\varepsilon) =(-r^2_1, r_1\varepsilon_1)$. After multiplication by $r_1^{2} > 0$, we obtain
\begin{equation}\label{(3.5)}
		\begin{array}{rcl}
			\dot{x} & = & r^2_1\left(
			X^{+}(x,-r_1^2,r_1\varepsilon_1\widetilde{\lambda})\phi_-(\varepsilon_1^2)+X^{-}(x,-r_1^2,r_1\varepsilon_1\widetilde{\lambda})(1-\phi_-(\varepsilon_{1}^{2}))\right),\\
			\dot{r_1} & = & -\frac{r_1}{2}\left(Y^{+}(x,-r_1^2,r_1\varepsilon\widetilde{\lambda})\phi_-(\varepsilon_1^2)+ Y^{-}(x,-r_1^2,r_1\varepsilon\widetilde{\lambda})(1-\phi_-(\varepsilon_1^2))\right),\\
			\dot{\varepsilon_1} & = & \frac{\varepsilon_1}{2}\left(Y^{+}(x,-r_1^2,r_1\varepsilon\widetilde{\lambda})\phi_-(\varepsilon_1^2)+ Y^{-}(x,-r_1^2,r_1\varepsilon\widetilde{\lambda})(1-\phi_-(\varepsilon_1^2))\right),	
		\end{array}
\end{equation}
with $\phi_-$ defined in Assumption \ref{assump-reg3}.

The line $r_1 = \varepsilon_1 = 0$ consists of singularities of \eqref{(3.5)} whose eigenvalues are given by $\left(0,-\frac{Y^-(x,0,0)}{2}, \frac{Y^{-}(x,0,0)}{2}\right)$. Following item (1) of Proposition \ref{prop-vv1}, we have $Y^{-}(x, 0,0)<0$ when $x>0$ and $Y^{-}(x, 0, 0) > 0$ when $x < 0$. Thus, in both the $VV_{1}$ and $VI_{3}$ cases, we deal with saddle-type singularities for $x\neq 0$, whose nonzero eigenvalues have the same magnitude, as in \cite{HK23,HHPY26}. 

For the phase directional chart $\overline{y}=1$, setting $(y,\varepsilon) =(r_2^2, r_2\varepsilon_2)$, after multiplication by $r_{2}^{2} > 0$ the extended system changes into
\begin{equation}\label{(3.6)}
		\begin{array}{rcl}
			\dot{x} & = & r^2_2\left(
			X^{+}(x,r_2^2, r_2\varepsilon_2\widetilde{\lambda})\phi_+(\varepsilon_2^2) +X^-(x,r_2^2,r_2\varepsilon_2\widetilde{\lambda})(1 -\phi_+(\varepsilon_2^2))\right),\\
			\dot{r_2} & = & \frac{r_2}{2}\left(Y^+(x,r_2^2,r_2\varepsilon_2\widetilde{\lambda})\phi_+(\varepsilon_2^2)+ Y^{-}(x,r_2^2, r_2\varepsilon_2\widetilde{\lambda})(1-\phi_+(\varepsilon_2^2))\right),\\
			\dot{\varepsilon_2} & = & -\frac{\varepsilon_2}{2}\left(Y^+(x,r_2^2, r_2\varepsilon_2\widetilde{\lambda})\phi_+(\varepsilon_2^2)+ Y^{-}(x,r_2^2, r_2\varepsilon_2\widetilde{\lambda})(1-\phi_+(\varepsilon_2^2))\right),
		\end{array}
\end{equation}
with $\phi_+$ defined in Assumption \ref{assump-reg3}.

For $x\neq 0$, the points on the edge of the cylinder of the form $(x,0,0)$ are again saddle-type singularities of \eqref{(3.6)}, with eigenvalues $\left(0,\frac{Y^+(x,0,0)}{2}, -\frac{Y^+(x,0,0)}{2}\right)$. Recall that $Y^{+}(x, 0,0)>0$ when $x>0$ and $Y^{+}(x, 0, 0) < 0$ when $x < 0$. 

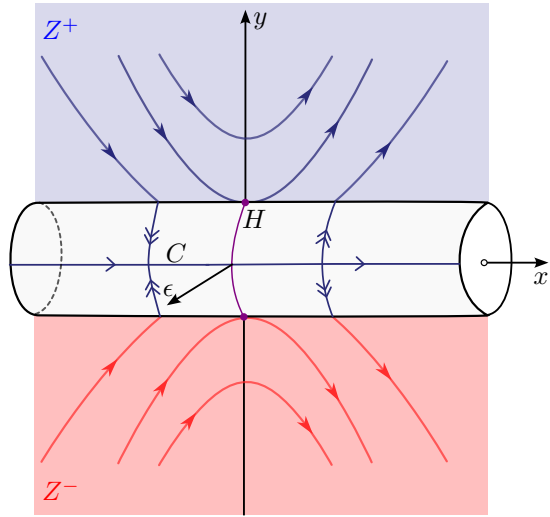
\begin{figure}[h!]
    \includegraphics[scale=0.4]{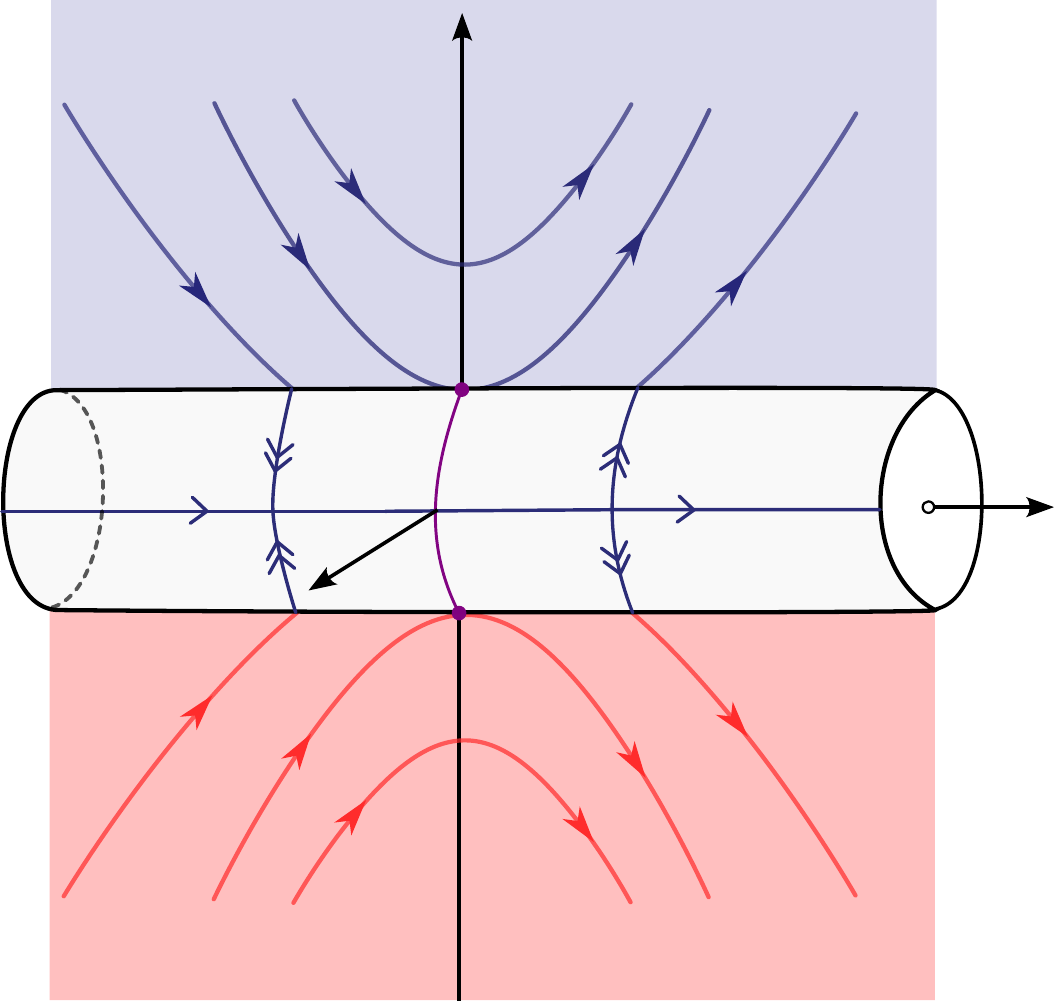}
    {\footnotesize \put(-5.1cm,2.9cm){$\epsilon$} \put(-4.05cm,3.8cm){$H$}
    \put(-5.06cm,3.36cm){$C$}
    \put(-6.7cm,6.3cm){\textcolor{blue}{$Z^+$}}
    \put(-6.7cm,0.2cm){\textcolor{red}{$Z^-$}
    }
    \put(-3.9cm,6.5cm){$y$}
    \put(-0.2cm,3.06cm){$x$}}
    \caption{Primary cylindrical blow-up in the case $VV_1$. 
    On the cylinder we find two critical manifolds $C$ and $H$.
    }
    \label{FirstBlowup}
\end{figure}

\subsection{Secondary cylindrical blow-up}

In this section we show that system \eqref{eq-psvf-reg-appendix} satisfies Assumptions T3–T6 of \cite{DD05} near the intersection $p_0$ of $H$ with $C$. We introduce the notation
\begin{equation}\label{eq-appendix-pwvfz}
Z^\pm(x,y,\varepsilon\widetilde{\lambda})=Z^\pm(x,y,0)+\varepsilon\widetilde{\lambda}\widetilde{Z}^\pm(x,y)+\mathcal{O}(\varepsilon^2),
\end{equation} 
where $Z^\pm=(X^\pm,Y^\pm)$, $\widetilde{Z}^{\pm}=(\widetilde{X}^{\pm},\widetilde{Y}^{\pm})$ and $\widetilde{\lambda}\sim 0$. Also, it is proven in \cite[Theorem 2.6]{BLM18} that the unfolding of the two-fold (which can be $VV_{1}$ or $VI_{3}$) is versal if \eqref{eq-appendix-pwvfz} satisfies the condition
	\begin{align}\label{2.7}
	    \widetilde{Y}^-\frac{\partial Y^+}{\partial x}\neq \widetilde{Y}^+\frac{\partial Y^-}{\partial x}
	\end{align}
at $(x,y,\varepsilon\widetilde{\lambda}) = (0,0,0)$.

We blow up the degenerate line $H$ of \eqref{eq-psvf-reg-appendix} to a cylinder by means of the secondary blow-up transformation
\begin{equation}\label{secondaryblowup}
(\widetilde{x},y_{2},\widetilde{\varepsilon},\rho) \longmapsto (\rho\widetilde{x},y_{2}, \rho\widetilde{\varepsilon})
= (x,y_{2},\varepsilon),
\end{equation}
where $(\widetilde{x},\widetilde{\varepsilon})\in \mathbb{S}^1$ and $\rho,\widetilde{\varepsilon}\geq 0$. As in the primary cylindrical blow-up, we work in different charts. We first consider the endpoints of the normally attracting and normally repelling parts of $C$ on the edge of the secondary cylinder.

\subsubsection{Dynamics in the phase directional charts $\widetilde{x}=\pm1$.}

In these directional charts, we verify Assumptions T3 and T4 of \cite{DD05}. We first consider the phase-directional chart $\widetilde{x}=1$, where $(x,\varepsilon)=(\rho,\rho\widetilde{\varepsilon})$. In these new coordinates, after division by $\rho>0$, the system becomes 
\begin{equation}\label{eq-2blowup-x=1}
\begin{array}{rcl}
    \dot{y}_2 & = & \frac{\partial Y^+}{\partial x}(0,0,0)\phi(y_2)+\frac{\partial Y^-}{\partial x}(0,0,0)(1-\phi(y_2))+\mathcal{O}(\rho,\widetilde{\varepsilon}), \\
    \dot{\rho} & = & \rho\widetilde{\varepsilon}^2\left(X^+(0,0,0)\phi(y_2)+X^-(0,0,0)(1-\phi(y_2)) + \mathcal{O}(\rho)\right), \\
    \dot{\tilde{\varepsilon}} & = & -\tilde{\varepsilon}^3\left(X^+(0,0,0)\phi(y_2) + X^-(0,0,0)(1-\phi(y_2)) + \mathcal{O}(\rho)\right).
\end{array}
\end{equation}

When $\rho=\widetilde{\varepsilon}=0$, system \eqref{eq-2blowup-x=1} has a singularity whose $y_{2}$-coordinate is given by
\begin{equation*}
    y_2 = y_{2c} := \phi^{-1}\left(\frac{-\frac{\partial Y^-}{\partial x}}{\frac{\partial Y^+}{\partial x} - \frac{\partial Y^-}{\partial x}}(0,0,0)\right).
\end{equation*}

The eigenvalues of the linearization of \eqref{eq-2blowup-x=1} computed at $(y_{2c},0,0)$ are given, respectively, by
\begin{equation*}
    \left(\frac{\partial Y^+}{\partial x} - \frac{\partial Y^-}{\partial x}\right)(0,0,0)\phi'(y_{2c}), \quad 0 \quad \text{and} \quad 0,
\end{equation*}
with the first eigenvalue being positive due to Proposition \ref{prop-vv1} and Assumption \ref{assump-reg2}. Therefore, this singularity admits two-dimensional center manifolds that are transverse to the unstable manifold, the latter being given by the $y_2$-axis. The endpoint of the repelling part of $C$ is normally hyperbolic, and hence Assumption T3 holds.

Moreover, each center manifold of system~\eqref{eq-2blowup-x=1} at the singularity $y_{2c}$ can be written as the graph of
\begin{align*}
    y_2=y_{2c}+\mathcal{O}(\rho,\widetilde{\varepsilon}).
\end{align*}

From the $(\rho,\widetilde{\varepsilon})$-component of \eqref{eq-2blowup-x=1}, we obtain the following center behavior:
\begin{equation*}
\dot{\rho}=\rho \widetilde{\varepsilon}^2\left(X^{sl}(0,0)+\mathcal{O}(\rho,\widetilde{\varepsilon})\right), \qquad \dot{\widetilde{\varepsilon}} = -\widetilde{\varepsilon}^3\left(X^{sl}(0,0)+\mathcal{O}(\rho,\widetilde{\varepsilon})\right).
\end{equation*}

Since $X^{sl}(0,0)>0$ in both the $VV_{1}$ and $VI_{3}$ cases (see Proposition~\ref{prop-vv1}), the system has, after division by $\widetilde{\varepsilon}^2$, an isolated hyperbolic saddle at $(\rho,\widetilde{\varepsilon})=(0,0)$. Hence Assumption T4 holds. Notice that the exponent in $\widetilde{\varepsilon}^2$ coincides with the order of degeneracy defined above.

The chart $\widetilde{x}=-1$ can be treated in a similar way by applying $(t,\rho,\widetilde{\varepsilon})\mapsto(-t,-\rho,-\widetilde{\varepsilon})$ to system \eqref{eq-2blowup-x=1}.

\subsubsection{Dynamics in the family chart $\widetilde{\varepsilon}=1$}

In this chart, we verify Assumptions T5 and T6 of \cite{DD05}. The reader may observe that, up to this point in the proof, the computations and conclusions are identical for the cases $VV_{1}$ and $VI_{3}$. However, in the family chart $\widetilde{\varepsilon}=1$, some differences arise, as we shall now see.

In this scaling chart, we have $x=\varepsilon x_2$. After division by $\varepsilon>0$ and passage to the limit $\varepsilon\to0$, we obtain the limiting system
\begin{align}\label{A.3}
    \begin{array}{cl}
       \dot{x}_2 & = X^+(0,0,0)\phi(y_2)
    +X^-(0,0,0)(1-\phi(y_2)),\\
    \dot{y}_2 & =\left(x_2 \frac{\partial Y^+}{\partial x}(0,0, 0)+\widetilde{\lambda}\widetilde{Y}^+(0,0)\right)\phi(y_2) + \left(x_2 \frac{\partial Y^-}{\partial x}(0,0,0)+\widetilde{\lambda}\widetilde{Y}^-(0,0)\right)(1-\phi(y_2)).
    \end{array}    
\end{align}

Observe that in the case $VV_{1}$ the first line of \eqref{A.3} is always positive. Indeed, $X^{\pm}(0,0,0)>0$ by \eqref{eq-prop-vv1}, and any convex combination of these two quantities is therefore positive. Hence, the system has no singularities at the top of the secondary cylinder. In the case $VI_{3}$, it is not difficult to see from \eqref{eq-prop-vi3} and \eqref{VI3-extra-condition} that, for $\widetilde{\lambda}=0$, system \eqref{A.3} has a center at
$$(x_2,y_2)=\left(0,\phi^{-1}\left( \frac{-X^{-}}{X^{+}-X^{-}}
(0,0,0)\right)\right).$$

A crucial observation is that, in both cases, when $\widetilde{\lambda}=0$, system \eqref{A.3} has an invariant line $\gamma$ defined by $\gamma = \{y_2 = y_{2c}\}$. Moreover, the dynamics of \eqref{A.3} restricted to $\gamma$ is given by
\begin{equation*}
\dot{x}_2 = X^{sl}(0,0)>0, \qquad \dot{y}_2 = 0.    
\end{equation*}

The line $\gamma$ is a
heteroclinic connection on the cylinder connecting the end point of the attracting part of $C$ to the end point of the repelling part of $C$. Thus, Assumption T5 holds (see Fig. \ref{secondBlowup}). In the case $VI_{3}$, the line $\gamma$ lies above the center, by \eqref{eq-prop-vi3} and \eqref{VI3-extra-condition}.

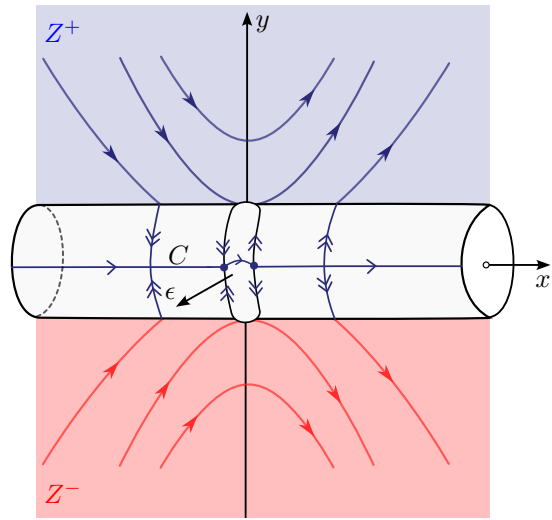
\begin{figure}[h!]
    \includegraphics[scale=0.4]{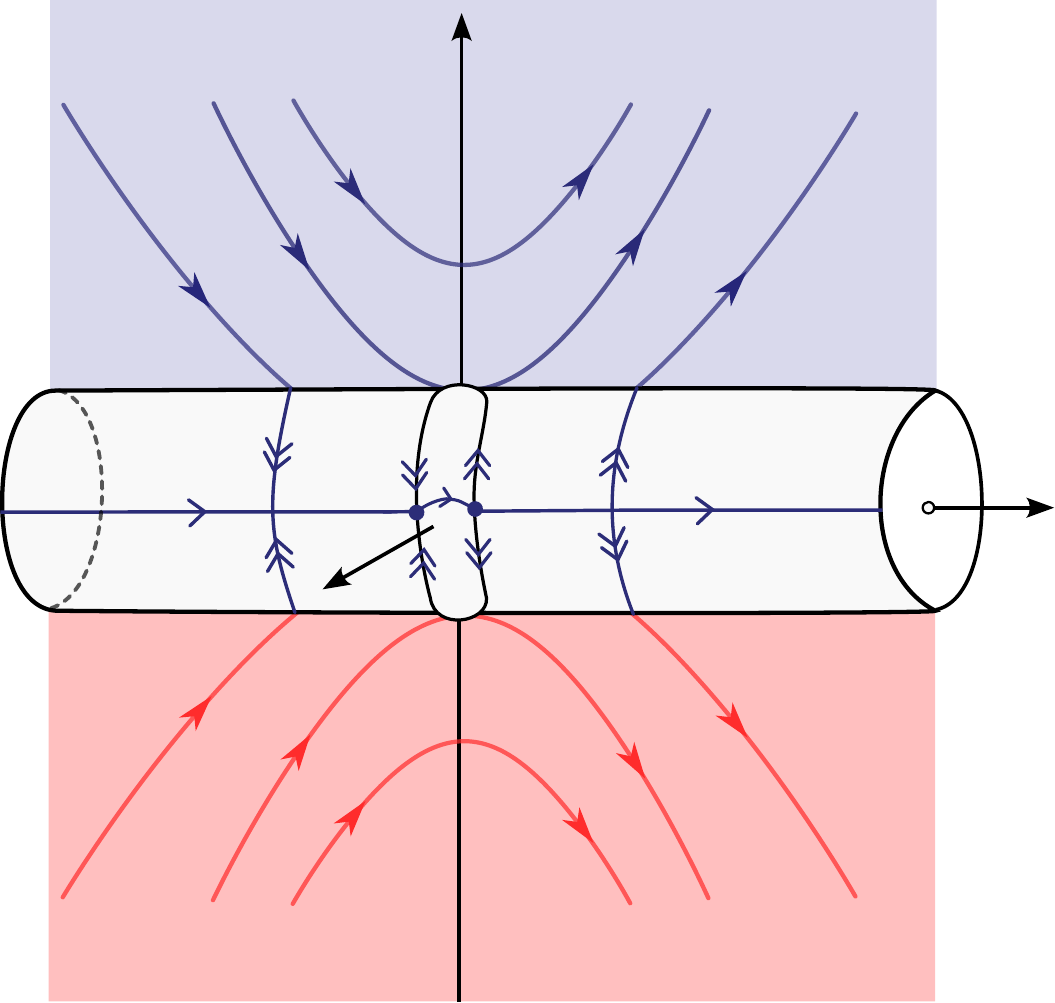}
    {\footnotesize \put(-5.1cm,2.9cm){$\epsilon$} 
    \put(-5.06cm,3.36cm){$C$}
    \put(-6.7cm,6.3cm){\textcolor{blue}{$Z^+$}}
    \put(-6.7cm,0.2cm){\textcolor{red}{$Z^-$}
    }
    \put(-3.9cm,6.5cm){$y$}
    \put(-0.2cm,3.06cm){$x$}}
    \caption{Secondary cylindrical blow-up in the case $VV_1$. We blow up $H$ through the cylindrical blow-up transformation \eqref{secondaryblowup}. 
    }
    \label{secondBlowup}
\end{figure}


Finally, it remains to verify Assumption T6, namely, that the invariant line $\gamma$ breaks regularly as the parameter $\widetilde{\lambda}$ varies near $0$. This was proved in \cite{HK23} for the case $VI_3$, using \eqref{eq-regularity} or, equivalently, \eqref{2.7}. 
The proof in the case $VV_1$ is analogous, and we therefore omit the details.

Under the assumptions of Theorem~\ref{mthm-singularity}, conditions T3--T5 hold in both cases. Therefore, the results of \cite{DD08} apply (see also \cite{HHPY26}).


\begin{thebibliography}{99}


\bibitem{BLM18}
     \newblock C. Bonet-Reves, J. Larrosa and T. M-Seara,
     \newblock Regularization around a generic codimension one fold-fold singularity,
     \newblock \emph{J. Differ. Equ.} \textbf{265} (2018), 1761--1838.

\bibitem{BPT2013}
    \newblock C.A. Buzzi, C. Pessoa, J. Torregrosa,
    \newblock Piecewise linear perturbations of a linear center,
    \newblock \emph{Discrete and Continuous Dynamical Systems} \textbf{33}(9) (2013), 3915--3936.

\bibitem{BST06}
    \newblock C.A. Buzzi, P.R. Silva, M.A. Teixeira.
    \newblock A singular approach to discontinuous vector fields on the plane,
    \newblock \emph{J. Differ. Equ.} \textbf{231} (2006), 633--655.


\bibitem{CFS2021}
    \newblock V. Carmona, F. Fernández-Sánchez, 
    \newblock Integral characterization for Poincar\'e half-maps in planar linear systems,
    \newblock \emph{J. Differ. Equ.} \textbf{305}, (2021) 319–346.


\bibitem{CFSN2021}
    \newblock V. Carmona, F. Fern\'andez-S\'anchez, and D. D. Novaes, 
    \newblock A new simple proof for Lum–Chua’s conjecture,
    \newblock \emph{Nonlinear Analysis: Hybrid Systems} \textbf{40} (2021), 100992.

\bibitem{CFSN2023} 
    \newblock V. Carmona, F. Fernández-Sánchez, D.D. Novaes,
    \newblock A succinct characterization of period annuli in planar piecewise linear differential systems with a straight line of nonsmoothness,
    \newblock \emph{J. Nonlinear Sci.} \textbf{33} (2023), 88.

\bibitem{CFSN2023ii} 
    \newblock V. Carmona, F. Fernández-Sánchez, D.D. Novaes,
    \newblock Uniqueness and stability of limit cycles in planar piecewise linear differential systems without sliding region.
    \newblock \emph{Commun. Nonlinear Sci. Numer. Simul.} \textbf{123} (2023), 107257.

\bibitem{CFSN2023iii} 
    \newblock V. Carmona, F. Fernández-Sánchez, D.D. Novaes,
    \newblock Uniform upper bound for the number of limit cycles of planar piecewise linear differential systems with two zones separated by a straight line,
    \newblock \emph{Applied Mathematics Letters} \textbf{137}, (2023), 108501.

\bibitem{Cau2012} 
    \newblock M. Caubergh,
    \newblock Hilbert’s sixteenth problem for polynomial Liénard equations,
    \newblock \emph{Qual. Theory Dyn. Syst.} \textbf{11}(1) (2012), 3--18.

  



\bibitem{DD05}
     \newblock P. De Maesschalck and F. Dumortier,
     \newblock Time analysis and entry-exit relation near planar turning points,
     \newblock \emph{J. Differ. Equ.} \textbf{215} (2005), 225--267.
     
\bibitem{DD08}
     \newblock P. De Maesschalck and F. Dumortier,
     \newblock Canard cycles in the presence of slow dynamics with singularities,
     \newblock \emph{Proc. Roy. Soc. Edinburgh Sect. A} \textbf{138} (2008), 265--299.

\bibitem{DMDR} P. De Maesschalck, F. Dumortier, R. Roussarie. \emph{Canard cycles: from birth to transition}, volume 73 of Ergebnisse der Mathematik und ihrer Grenzgebiete. 3. Folge. A Series of Modern Surveys in Mathematics [Results in Mathematics and Related Areas. 3rd Series. A Series of Modern Surveys in Mathematics]. Springer, Cham, (2021).

\bibitem{DMH2015}
     \newblock P. De Maesschalck and R. Huzak,
     \newblock Slow divergence integrals in classical Liénard equations near centers,
     \newblock \emph{J. Dyn. Differ. Equ.} \textbf{27}(1) (2015), 177--185.

\bibitem{DMHP}
     \newblock P. De Maesschalck, R. Huzak and O.H. Perez,
     \newblock Canard cycles of non-linearly regularized piecewise smooth vector fields,
     \newblock \emph{J. Differ. Equ.} \textbf{460} (2026), 114079.

\bibitem{DPR2007}
    \newblock F. Dumortier, D. Panazzolo and R. Roussarie,
    \newblock More limit cycles than expected in Liénard equations, 
    \newblock \textit{Proc. Am. Math. Soc.}, \textbf{135}(6) (2007), 1895–1904.

\bibitem{DRR94}
    \newblock F. Dumortier, R. Roussarie, and C. Rousseau,
    \newblock Hilbert’s 16th problem for quadratic vector fields, 
    \newblock \textit{J. Differ. Equ.}, \textbf{110}(1) (1994), 86--133.

\bibitem{Filippov}
    \newblock A.F. Filippov,
    \newblock Differential Equations with Discontinuous Right-Hand Sides, Mathematics and Its Applications (Soviet Series),
    \newblock \emph{Kluwer Academic Publishers}, Dordrecht, 1988.

\bibitem{FPRT1998}
\newblock E. Freire, E. Ponce, F. Rodrigo, and F. Torres,
    \newblock Bifurcation sets of continuous piecewise linear systems with two zones,
    \newblock \emph{Int. J. Bifurcation Chaos
Appl. Sci. Eng}, \textbf{8}(11) (1998), 2073--2097.


\bibitem{GRS2024}
\newblock A. Gasull, G. Rond\'on and P.R. da Silva,
    \newblock On the number of limit cycles for piecewise polynomial holomorphic systems,
    \newblock \emph{SIAM Journal on Applied Dynamical Systems}, \textbf{23}(3) (2024).

    \bibitem{GST2011}
    \newblock M. Guardia, T.M. Seara and M.A. Teixeira,
    \newblock Generic bifurcations of low codimension of planar Filippov systems,
    \newblock \emph{J. Differ. Equ.}, \textbf{250}(4) (2011), 1967--2023.

\bibitem{HY2012}
    \newblock S. M. Huan and X.S. Yang,
    \newblock On the number of limit cycles in general planar piecewise linear systems,
    \newblock \emph{Discrete Contin. Dyn. Syst.}, \textbf{32}(6) (2012), 2147--2164.


    \bibitem{Hil1900}
\newblock D. Hilbert,
\newblock Mathematical Problems,
\newblock \textit{Bull. Amer. Math. Soc.}  \textbf{8}(10) (1902), 437--479.

\bibitem{HHPY26}
\newblock J. Huang, R. Huzak, O.H. Perez and J. Yao,
     \newblock Cyclicity of sliding cycles with singularities of regularized piecewise smooth visible-invisible two-folds,
     \newblock \emph{J. Differ. Equ.}, \textbf{465} (2026), 114205.

 \bibitem{HDM-generalized}
     \newblock R. Huzak and P. De Maesschalck,
     \newblock Slow divergence integrals in generalized Li\'enard
equations near centers,
     \newblock \emph{Electron. J. Qual. Theory Differ. Equ.}, \textbf{66} (2014), 1--10.    

\bibitem{HK23}
     \newblock R. Huzak and K. Uldall Kristiansen,
     \newblock The number of limit cycles for regularized piecewise polynomial systems is unbounded,
     \newblock \emph{J. Differ. Equ.}, \textbf{342} (2023), 34--62.

\bibitem{HK24}
     \newblock R. Huzak and K. Uldall Kristiansen,
     \newblock Sliding cycles of regularized piecewise linear visible-invisible twofolds,
     \newblock \emph{Qual. Theory Dyn. Syst.}, \textbf{23}(256) (2024).

\bibitem{HP}
     \newblock R. Huzak and O.H. Perez,
     \newblock An unbounded number of canard limit cycles in linear regularizations of piecewise linear systems,
     \newblock \emph{J. Nonlinear Science} \textbf{36}(61) (2026).


\bibitem{KRG03}
     \newblock Yu. A. Kuznetsov, S. Rinaldi and A. Gragnani,
     \newblock One parameter bifurcations in planar Filippov systems,
     \newblock \emph{Int. J. Bifur. Chaos} \textbf{13}(8) (2003), 2157--2188.


\bibitem{LliPon2013} \newblock J. Llibre and E. Ponce,
\newblock Three nested limit cycles in discontinuous piecewise
linear differential systems with two zones,
\newblock \emph{Dyn. Contin. Discrete Impuls.
Syst., Ser. B, Appl. Algorithms} \textbf{19}(3) (2012), 325--335.


\bibitem{LumChu1991} \newblock R. Lum and L.O. Chua,
\newblock Global properties of continuous piecewise linear vector fields. Part I: Simplest case in $\mathbb{R}^{2}$,
\newblock \emph{International Journal of Circuit Theory and Applications} \textbf{19}(3) (1991), 251--307.

\bibitem{PanSil2017} \newblock D. Panazzolo and  P.R. da Silva,
\newblock Regularization of discontinuous foliations: Blowing up and sliding conditions via Fenichel Theory,
\newblock \emph{J. Differ. Equ.} \textbf{263}(12) (2017), 8362--8390.

\bibitem{PRS2023} \newblock O.H. Perez, G. Rond\'on and  P.R. da Silva,
\newblock Slow-fast normal forms arising from piecewise smooth vector fields,
\newblock \emph{J. Dyn. Control Syst.} \textbf{29}(4) (2023), 1709--1726.

\bibitem{Smale2000}
    \newblock S. Smale,
    \newblock Mathematical problems for the next century.
    \newblock In Mathematics:
frontiers and perspectives, pages 271--294. Amer. Math. Soc., Providence, RI, 2000.

\bibitem{SM2002} \newblock J. Sotomayor and A.L.F Machado,
\newblock Structurally stable discontinuous vector fields in the plane,
\newblock \emph{Qual. Theory Dyn. Syst.} \textbf{3} (2002), 227--250.

\bibitem{ST96}
     \newblock J. Sotomayor and M.A. Teixeira,
     \newblock Regularization of discontinuous vector fields,
     \newblock \emph{In Proceedings of the International Conference on Differential Equations}, Lisboa, (1996), 207--223.

\end{thebibliography}
\end{document}